\documentclass[11pt]{amsart}
\usepackage{amsmath,amsthm, amscd, amssymb, amsfonts, mathtools}
\usepackage{tikz-cd}
\usepackage[all]{xy}
\usepackage{mathrsfs}
\usepackage[inline]{enumitem}
\usepackage[T2A,T1]{fontenc}
\usepackage{bigstrut}
\usepackage{rotating}
\usepackage{hyperref}
\usepackage{multirow}
\usepackage{xcolor}

\usepackage{caption}

\usepackage[ansinew]{inputenc}
\usepackage{graphicx,fancyhdr}

\newcommand{\xrightarrowdbl}[2][]{%
\xrightarrow[#1]{#2}\mathrel{\mkern-14mu}\rightarrow
}

\newcommand{\comment}[1]{}

\numberwithin{equation}{section}
\theoremstyle{plain}
\newtheorem{maintheorem}{Theorem}
\newtheorem{theorem}{Theorem}[section]
\newtheorem{lemma}[theorem]{Lemma}
\newtheorem{coro}[theorem]{Corollary}

\newtheorem{prop}[theorem]{Proposition}

\theoremstyle{definition}
\newtheorem{definition}[theorem]{Definition}
\newtheorem{example}[theorem]{Example}

\theoremstyle{remark}

\newtheorem{remark}[theorem]{Remark}

\newtheorem{step}{Case}

\def\pf{\begin{proof}}
\def\epf{\end{proof}}

\newcommand{\mfL}{\mathfrak{g}}
\newcommand{\mgo}{\mathfrak{m}}
\newcommand{\mfG}{\mathfrak{fsl}(2)}
\newcommand{\spl}{\mathfrak{sl}}
\newcommand{\ugo}{\mathfrak u}

\newcommand{\nucleouno}{\mathbf R}
\newcommand{\nucleodos}{\mathfrak R}
\newcommand{\nucleo}{\mathbf T}

\newcommand{\ku}{ \Bbbk}
\newcommand{\kut}{ \ku^{\times}}

\newcommand{\I}{\mathbb I}

\newcommand{\J}{\mathbb J}
\newcommand{\N}{\mathbb N}

\newcommand{\jordan}{{\mathtt{J}}}

\newcommand{\Ss}{{\mathcal S}}

\newcommand{\cV}{\mathcal{V}}

\newcommand{\cO}{\mathcal{O}}

\newcommand{\Att}{\mathtt A}
\newcommand{\Btt}{\mathtt B}
\newcommand{\Ctt}{\mathtt C}

\newcommand{\Ftt}{\mathtt F}
\newcommand{\mtt}{\mathtt  m}
\newcommand{\Utt}{\mathtt U}
\newcommand{\Vtt}{\mathtt V}
\newcommand{\Wtt}{\mathtt W}

\newcommand{\cP}[3]{\mathcal{P}_{#1,#2}(#3)}

\newcommand\ad{\operatorname{ad}}

\newcommand{\Aut}{\operatorname{Aut}}

\newcommand{\car}{\operatorname{char}}
\newcommand{\Der}{\operatorname{Der}}

\newcommand{\End}{\operatorname{End}}
\newcommand{\id}{\operatorname{id}}
\newcommand{\Ind}{\operatorname{Ind}}
\newcommand{\Jac}{\operatorname{Jac}}

\newcommand{\Hom}{\operatorname{Hom}}

\newcommand{\soc}{\operatorname{soc}}
\newcommand{\rad}{\operatorname{rad}}
\newcommand{\res}{\operatorname{Res}}

\newcommand{\Indec}{\operatorname{indec}\nolimits}
\newcommand\ext{\operatorname{Ext}}

\newcommand{\lmod}[1]{\hspace{-2pt}{}_{#1}\mathcal{M}}
\newcommand{\rmod}[1]{\mathcal{M}_{ #1 }}
\newcommand\rep{\operatorname{rep}}

\newcommand{\yd}[1]{{}^{ #1 }_{ #1 }\mathcal{YD}}

\newcommand{\cD}{\mathcal{D}}

\newcommand{\toba}{\mathscr{B}}
\newcommand{\ot}{\otimes}

\newcommand{\Irr}{\operatorname{irrep}}

\allowdisplaybreaks

\newcounter{tabla}\stepcounter{tabla}

\begin{document}

\title[Double of the restricted Jordan plane in characteristic $2$]{On the Drinfeld double  of  the restricted Jordan plane in characteristic $2$}

\author[Andruskiewitsch, Bagio, Della Flora, Fl\^ores]
{Nicol\'as Andruskiewitsch, Dirceu Bagio, Saradia Della Flora, Daiana Fl\^ores}

\address{N. A.: FaMAF-Universidad Nacional de C\'ordoba, CIEM (CONICET),
Medina A\-llen\-de s/n, Ciudad Universitaria, 5000 C\' ordoba, Argentina.} \email{nicolas.andruskiewitsch@unc.edu.ar}

\address{D.B.: Departamento de Matem\'atica, Universidade Federal de Santa Catarina,
88040-900, Florianópolis, SC, Brazil}
\email{d.bagio@ufsc.br}

\address{S. D. F., D. F.: Departamento de Matem\'atica, Universidade Federal de Santa Maria,
97105-900, Santa Maria, RS, Brazil} \email{saradia.flora@ufsm.br, flores@ufsm.br}

\thanks{\noindent 2010 \emph{Mathematics Subject Classification.}
16T20, 17B37. \newline The work of N. A. was partially supported by CONICET(PIP 11220200102916CO),
FONCyT-ANPCyT(PICT-2019-03660) and Secyt (UNC)}

\begin{abstract}
We consider the restricted Jordan plane  in  characteristic $2$, a  finite-dimensional Nichols algebra 
quotient of the Jordan plane that was introduced by Cibils, Lauve and Witherspoon. We
extend results from \texttt{arXiv:2002.02514} on the analogous object in odd characteristic.
We show that the Drinfeld double  of the restricted Jordan plane fits into an exact sequence of Hopf algebras
whose kernel is a normal local commutative Hopf subalgebra and the cokernel is the restricted enveloping algebra of a restricted
Lie algebra $\mathfrak m$ of dimension 5.
We show that $\mathfrak u(\mathfrak m)$ is tame and compute explicitly
the indecomposable modules.
An infinite-dimensional Hopf algebra covering the Drinfeld double of the restricted Jordan plane is 
introduced. Various quantum Frobenius maps are described.  
\end{abstract}

\maketitle

\setcounter{tocdepth}{1}
\tableofcontents

\section{Introduction}
\subsection{Blocks and their Nichols algebras}
Let $\ku$ be an algebraically closed field. A braided vector space is a pair $(V, \mathfrak{c})$ where $V$ is 
a vector space and $\mathfrak{c}\in GL (V\otimes V)$ (called the braiding)
satisfies the braid equation 
\[(\mathfrak{c}\otimes \id)(\id \otimes \mathfrak{c})(\mathfrak{c}\otimes \id) =  (\id \otimes \mathfrak{c})(\mathfrak{c}\otimes \id)(\id \otimes \mathfrak{c}).\] 
Braided vector spaces and their Nichols algebras are key elements of the structure of Hopf algebras, cf. \cite{A-leyva}. Given $\epsilon \in \kut$ and $\ell \in \N_{\geq 2}$, let $\cV(\epsilon,\ell)$ 
be the braided vector space  with basis $\{x_i:  1 \leq  i \leq \ell\}$ and braiding 
$\mathfrak{c}_{\epsilon}$  given by
\begin{align}\label{eq:braidingsuperjordan}
\mathfrak{c}_{\epsilon}(x_i \otimes x_j) &= \begin{cases}
\epsilon x_1 \otimes x_i, &j =1,
\\
(\epsilon x_j + x_{j - 1}) \otimes x_i,& 2 \leq  j \leq \ell,
\end{cases} &  1 \leq  i \leq \ell.
\end{align}
The braided vector spaces $\cV(\epsilon,{\ell})$ are called \emph{blocks}; their Nichols algebras  are related to pointed Hopf algebras with abelian group.

\medbreak
The Jordan plane on  a field of characteristic different of $2$ is the well-known quadratic algebra  
\[J = \ku\langle x_1, x_2 : x_1x_2 - x_2x_1 - \frac{1}{2}x_1^2\rangle.\]
The super Jordan plane, introduced in  \cite{aah-triang}, is the 
graded  algebra 
\[sJ = \ku\langle x_1, x_2 : x_1^2, x_2 x_{21} - x_{21} x_2 - x_1 x_{21}\rangle,\]
where $x_{21} = x_2 x_1 + x_1 x_2$. 
Both the Jordan plane and the super Jordan plane
bear structures of  braided Hopf algebra  \cite{G,aah-triang}. 
When $\car \ku = 0$ the Nichols algebra $\toba(\cV(1,2)) $ is isomorphic to $J$ as braided Hopf algebra and similarly $\toba(\cV(-1,2)) \simeq sJ$.
These are the only  Nichols algebras of blocks with finite Gelfand-Kirillov dimension
\cite{aah-triang}.

\medbreak
When $\car \ku = p > 2$, it was shown in \cite{clw} that 
\begin{align*}
\toba(\cV(1,2)) &\simeq J/ \langle x_1^p, x_2^p\rangle& & \text{ and} & \dim \toba(\cV(1,2)) &= p^2.
\end{align*}
Motivated by this result, it was proved in \cite{aah-oddchar} under the same hypothesis that 
\begin{align*}
\toba(\cV(-1,2)) &\simeq sJ/ \langle x_{21}^p, x_2^{2p}\rangle& & \text{ and} & \dim \toba(\cV(-1,2)) &= 4p^2.
\end{align*}

The Nichols algebra $\toba(\cV(1,2))$, respectively  $\toba(\cV(-1,2))$,
is called the  \emph{restricted} Jordan, respectively super Jordan, plane. 
Restricted versions of other Nichols algebras from \cite{aah-triang} appear in \cite{aah-oddchar,ABDF}.

\subsection{Drinfeld doubles and exact sequences}
Assume that  $\car \ku = p$ is odd. Let $C_N$ denote  the cyclic group of order $N$. 
The Drinfeld doubles $D(\toba(\cV(1,2)) \# \ku C_p)$ and $D(\toba(\cV(-1,2)) \# \ku C_{2p})$  were studied in \cite{ap,ap1}.
It was shown in \cite{ap} that there is
a normal local commutative Hopf algebra $\nucleouno$ fitting into an exact sequence
\[\nucleouno \hookrightarrow D(\toba(\cV(1,2)) \# \ku C_p)\twoheadrightarrow \mathfrak u(\spl_2(\ku));\]
here $\mathfrak u(\spl_2(\ku))$
is the restricted enveloping algebra. Then the simple $\mathfrak u(\spl_2(\ku))$-modules are in bijective correspondence with those of $D(\toba(\cV(1,2)) \# \ku C_p)$.

\medbreak As for the restricted super Jordan plane, there is a Hopf superalgebra  $\cD$ such that
$D(\toba(\cV(-1,2)) \# \ku C_{2p})\simeq  \cD \# \ku C_2$.
Then there are a local commutative Hopf subalgebra $\nucleodos$ and
an exact sequence of Hopf superalgebras
\[\nucleodos \xhookrightarrow[]{\iota} \cD  \xrightarrowdbl[]{\pi} \ugo(\mathfrak{osp}(1|2)).\]
See \cite{ap1}.
As in the case of the restricted Jordan plane, the simple $\cD$-modules are in bijective correspondence with  those of $\ugo(\mathfrak{osp}(1|2))$.

\medbreak
Assume from now on that $\car\ku = 2$,
so that $\mathtt{V} \coloneqq  \cV(1,2) = \cV(-1,2)$. The Nichols algebra
$\toba(\mathtt{V})$ was also computed in \cite{clw}: it turns out that
$\toba(\mathtt{V})$ has dimension 16 and is presented by generators $x_1,x_2$ with defining relations
\begin{align} \label{eq-jordan-plane-char2}
x_1^2=0,&  &x_2^4=0,& &x_2^2x_1=x_1x_2^2+x_1x_2x_1,& &x_1x_2x_1x_2=x_2x_1x_2x_1.
\end{align}
Thus  $\toba(\mathtt{V})$, that we call the restricted Jordan plane in characteristic $2$, is closer to the restricted super Jordan plane
than to the Jordan one (in odd characteristic). 
Let $H = \toba(\mathtt{V}) \#\ku C_2$ be the  bosonization of the restricted Jordan plane and 
let $\mgo$ be the Lie algebra of dimension 5 described in \S \ref{subsection:restricted}.
To start with, we extend  \cite[Proposition 1.9]{ap} and \cite[Theorem 4.6]{ap1}.

\begin{maintheorem}\label{mainthm:abelian-extension}
The Drinfeld double $D(H)$ fits in an abelian extension
$$\nucleo\overset{\iota}\hookrightarrow D(H)\overset{\pi}\twoheadrightarrow \ugo(\mgo),$$ 
where $\nucleo$ is a normal local commutative Hopf subalgebra of $D(H)$.
\end{maintheorem}

Hence the simple $D(H)$-modules are the same as those of 
$\ugo(\mgo)$, which are known \cite{do}: the trivial module $V_0$ and a 3-dimensional module $V_1$. 
Even more the functor $\pi^*: \lmod{\ugo(\mgo)} \to \lmod{D(H)}$ is tensor and preserves 
indecomposability.
We are lead to study the representation theory of $\ugo(\mgo)$. We prove:

\begin{maintheorem}\label{mainthm:indec}
The algebra $\ugo(\mgo)$ is  tame. The  indecomposable $\ugo(\mgo)$-modules
are the simple modules $V_0$ and $V_1$, their projective covers, 8 families of string modules 
and 2 families of band modules.
\end{maintheorem}

See Theorem \ref{thm:clasif-indec-um} for the precise description of these modules.
The paper is organized as follows. The Section \ref{section:double-restrictedJP}
contains preliminaries on restricted enveloping algebras, the presentation of $D(H)$ and the proof of 
Theorem \ref{mainthm:abelian-extension}, see Theorem \ref{teo:sequence}. 
Section \ref{sec:tame} is devoted to the basics of the representation theory
of $\ugo(\mgo)$; here we prove that $\ugo(\mgo)$ is tame
and that the associated basic algebra is special biserial.  This last property allows us 
to compute the indecomposable $\ugo(\mgo)$-modules in Section \ref{sec:indec}.
In Section \ref{sec:double-Jordan} we discuss an infinite-dimensional analogue of $D(H)$, extending results from
\cite{ap} and \cite{ap1}.

\subsection{Notations and conventions}
We denote the natural numbers by $\N$, and $\N_0=\N\cup \{0\}$. 
For $k<\ell \in\N_0$, we set $\I_{k,\ell}=\{n \in \N_0: k \leq n \leq \ell\}$ and $\I_{\ell}=\I_{1, \ell}$.
We fix an algebraically closed field $\ku$ of characteristic 2.
All vector spaces, algebras and tensor products are over $\ku$. 
If $U$ is a vector space and $S \subseteq U$, then $\ku S$ stands for the linear subspace of $U$ generated by $S$.
Given $\lambda \in \ku$, 
we denote by $\jordan_{r}(\lambda)$ the Jordan block of size $r$ associated to  $\lambda$. 

Let $\ku G$, respectively $\ku^G$, denote the group algebra of a finite group  $G$, respectively
the algebra of functions on $G$; thus  $(\ku G)^* \simeq\ku^G$. If $S \subseteq G$, 
then $\langle S\rangle$ denotes the subgroup generated by $S$. 

Given an algebra $A$, let $\lmod{A}$, respectively  $\rmod{A}$, denote  the category of  finite-dimensional left,
respectively right, $A$-modules. 
Also let $\Irr A$, respectively $\Indec A$, denote the set of isomorphism classes of  simple, respectively 
indecomposable, objects  in $\lmod{A}$.
We often denote indistinctly a class in $\Irr A$, or $\Indec A$, and one of its representatives.
The Jacobson radical of $A$ is denoted by $\Jac A$ and the radical of  $V\in \lmod{A}$ is denoted by $\rad V$.
If $S \subseteq A$, then $\langle S\rangle$ denotes the two-sided ideal generated by $S$. 

Given $L$ a Hopf algebra with  antipode  $\Ss$ (assumed always bijective), the group of  group-like elements in $L$ is denoted by 
$G(L)$. For $g,h\in G(L)$, the linear space of $(g,h)${\it -primitive  (or skew primitive)} elements
is  
$$
\cP{g}{h}{L}\coloneqq \{x\in L\,:\,\Delta(x)=  x\ot h + g\ot x\}.$$ 
As usual,  $\mathcal{P}(L) \coloneqq  \cP{1}{1}{L}$. 
The category of Yetter-Drinfeld modules over $L$ is denoted by $\yd{L} $, see e.~g. \cite[\S 11.6]{Ra}.
We refer to \cite{Ra} for unexplained terminology.

\section{The Drinfeld double of the  restricted Jordan plane}\label{section:double-restrictedJP}

\subsection{The restricted Lie algebra \texorpdfstring{$\mgo$}{}} \label{subsection:restricted} 
Recall that a {\it restricted Lie algebra } (over our field $\ku$ of characteristic 2)
is a pair $(\mfL, (\,\,)^{[2]})$, where $\mfL$ is a Lie algebra 
and $(\,\,)^{[2]}:\mfL\to \mfL$ is a map, called the {\it $2$-operation}, that satisfies
\begin{align*}
(\lambda x)^{[2]} &= \lambda^2 x^{[2]}, &
\ad_{x^{[2]}} &= (\ad_x)^2, &
(x+y)^{[2]} &= x^{[2]}+y^{[2]}+ [x,y],
\end{align*}
for all $x,y\in \mfL$, $\lambda\in \ku$.     
The {\it restricted enveloping algebra} of  $(\mfL, (\,\,)^{[2]})$ is 
\[\mathfrak{u}(\mfL)\coloneqq  U(\mfL)/\langle x^2-x^{[2]}\,:\,x\in \mfL\rangle,\]
where $U(\mfL)$ is the  enveloping algebra of $\mfL$.
The Hopf algebra $\mathfrak{u}(\mfL)$ is pointed  and involutory (that is, $\Ss^2=\id$), thus it is spherical.

Given a Lie subalgebra $\mathfrak{e}$  of $\mfL$, the smallest restricted Lie subalgebra of $\mfL$ containing $\mathfrak{e}$ is denoted by $\mathfrak{e}_{[2]}$. 

Let $\mathfrak{l}$ be a Lie algebra. A {\it $2$-envelope of $\mathfrak{l}$} is 
a triple $(\mfL, (\,\,)^{[2]},\iota)$ consisting of a restricted Lie algebra $(\mfL, (\,\,)^{[2]})$ and an injective Lie algebra map $\iota:\mathfrak{l}\to \mfL$ such that $\iota(\mathfrak{l})_{[2]}=\mfL$. 
A $2$-envelope is  {\it minimal} if the center $Z(\mfL)$ of $\mfL$ is contained in $Z(\iota(\mathfrak{l}))$. 
For any Lie algebra there exists a  minimal $2$-envelope unique up to Lie algebra isomorphism; see \cite[Theorem 1.1.6]{str}.

\medbreak
Let $\mfG\coloneqq W(1,\underline{2})^{'}$ be the derived Lie algebra of the $4$-dimensional  Witt Lie algebra,
see \cite{s}. The notation $\mfG$ is taken from \cite{cushingetal}  where this Lie algebra  is called the fake 
$\mathfrak{sl}(2)$. 
Then $\mfG$ has a basis $\{a,b,c\}$ and bracket
\begin{align}\label{lie-bracket}
&[a,b]=c,& &[a,c]=a,& &[b,c]=b.&
\end{align}
It is known that $\mfG$ is the unique, up to isomorphism, simple Lie algebra of dimension $3$ \cite{str}. The center of 
$U(\mfG)$ and $\Irr U(\mfG)$ were computed in \cite{do}. 

\medbreak
The Lie algebra $\mfG$ is not restricted. The minimal $2$-envelope of $\mfG$ is the Lie algebra $\mgo$ with basis $\{b',b,c,a,a'\}$, bracket
\eqref{lie-bracket} and
\begin{align*}
&\begin{aligned}
&[a',b]=a,&  &[a',b']=c, & &[a,b']=b,&  
&[a',a]=[a',c]=[b',b]=[b',c]=0;
\end{aligned}
\end{align*}
and $2$-operation $(\,\,)^{[2]}:\mgo\to \mgo$ given by
\begin{align*}
&(a')^{[2]}=(b')^{[2]}=0,& &c^{[2]}=c,& &a^{[2]}=a',& &b^{[2]}=b'.&
\end{align*}

\begin{remark} \label{rem:adjoint-representation}
$\mathfrak{u}(\mgo) \simeq \ku\langle a,b,c\rangle / I$ where $I$ is generated by the relations
\begin{align}\label{rel-u(L)}
&ab+ba=c,\!& &ac+ca=a,\!& &bc+cb=b,\!& &a^4=b^4=0,\!&&c^2+c=0.&
\end{align}
That is,  
\begin{align}\label{eq:u(m)-iso}
\mathfrak{u}(\mgo) \simeq U(\mfG) / \langle a^4, b^4, c^2+c\rangle.
\end{align} 

Moreover, the adjoint representation $\ad: U(\mfG)  \to \End \mfG$ factorizes through $\mathfrak{u}(\mgo) $.
One has $\mfG = [\mgo, \mgo]$; in fact the restricted Lie algebra $\mgo$ is isomorphic to $\Der(\mfG)$.
The automorphism group of the restricted Lie algebra  $\mgo$  is isomorphic to $SL(2, \ku)$ via
\begin{align*}
&SL(2, \ku) \ni g= \left(\begin{matrix}
\kappa & \lambda \\  \mu & \zeta
\end{matrix}\right) \mapsto \varphi_g \in \Aut (\mgo), \\
&\begin{aligned}
\varphi_g (a) &= \kappa a + \lambda b, &
\varphi_g (b) &= \mu a + \zeta b, & 
\varphi_g (c) &= c, \\
\varphi_g (a') &= \kappa^2 a' + \lambda^2 b' +\kappa \lambda  c, &
\varphi_g (b') &= \mu^2 a' + \zeta^2 b' + \mu \zeta c.
\end{aligned}
\end{align*}

\end{remark}

For later use we collect  some consequences of the defining relations.

\begin{lemma}\label{relbasic}
The following relations hold in $\mathfrak{u}(\mgo)$:
\begin{align*}
& ba = ab+c, & & ba^2 = a^2b+a, & & ba^3 = a^3b+a^2\overline{c}  \\
& b^2a =ab^2+b, & & b^2a^2 =a^2b^2+c, & & b^2a^3 = a^3b^2+a^2b+a\overline{c},\\
&b^3a = ab^3+b^2\overline{c}, & & b^3a^2 = a^2b^3+ab^2+bc, & & b^3a^3 = a^3b^3+a^2b^2c+abc,\\
&ca=ac+a, &&  ca^2 =a^2c, && ca^3 =a^3\overline{c}, \\
&cb=bc+b, &  & cb^2=b^2c, & & cb^3 = b^3\overline{c}, 
\end{align*}
where $\overline{c}=1+c$.\qed
\end{lemma}

\subsection{ The bosonization of the restricted Jordan plane}
Let $\Gamma=\langle g\rangle$ be the cyclic group of order $2$, written multiplicatively. 
Let $V \in \yd{\ku \Gamma}$ with a basis $\{x_1,x_2\}$, action  $\cdot$ and coaction $\delta$ given by
\begin{align}\label{YD-structure}
g\cdot x_1=x_1,& &g\cdot x_2=x_1+x_2,& &\delta(x_i)=g\otimes x_i,& &i \in \I_2.
\end{align}
Thus $V\simeq \cV(1,2)$ as braided vector spaces.
As said,
$\toba(V)$  is presented by generators $x_1,x_2$ and relations \eqref{eq-jordan-plane-char2}, see \cite{clw}.
Let $x_{21}\coloneqq x_1x_2+x_2x_1 \in\toba(V)$. Then $\toba(V)$ has a PBW-basis 
\[\{x_1^{m_1}x_{21}^{m_2}x_2^{n}\,:\,m_i\in\I_{0,1},\,i \in \I_2, \, n\in \I_{0,3}\},\]
and $\dim \toba(V)= 2^4$.
As in \cite{ap}, $\toba(V)$ is called {\it the restricted Jordan plane}. 

\begin{lemma}\cite[Corollary 3.4]{clw}\label{lem-boso-restricted}
The bosonization $H\coloneqq \toba(V)\#\ku\Gamma$ is a pointed Hopf algebra of dimension $2^5$ generated by $x_1,x_2,g$ with relations \eqref{eq-jordan-plane-char2} and
\begin{align} \label{rel-boson}
&gx_1=x_1g,& &gx_2=x_2g+x_1g,& &g^2=1.&
\end{align}
The coproduct of $H$ is given by
\begin{align}\label{for-copro-boson}
&\Delta(g)=g\otimes g,& \hspace{20pt} &\Delta(x_i)=x_i\otimes 1+g\otimes x_i,& &i\in \I_2.& &\hspace{30pt} &\qed
\end{align}
\end{lemma}

\subsection{The double}\label{subsection:rest-JP} 
We proceed to present  the Drinfeld double $D(H)$  by ge\-ne\-ra\-tors and relations. To fix the notation we recall that
the Drinfeld double $D(L)$ of a finite-dimensional Hopf algebra $L$ is the coalgebra  $L\otimes L^{\ast\,\rm{op}}$ 
with the multiplication and the antipode of $D(L)$  given respectively by
\begin{align*}
&(h\bowtie f)(h'\bowtie f')=\langle f_{(1)},h'_{(1)}\rangle\langle f_{(3)},\Ss(h'_{(3)})\rangle(hh'_{(2)}\bowtie f'f_{(2)}),\\
&\Ss(h\bowtie f)=((1\bowtie \Ss^{-1}(f))(\Ss(h)\bowtie \varepsilon)),
\end{align*} 
where the multiplication $ff'$ is considered in $L^{\ast}$ rather than in $L^{\ast\,\rm{op}}$. 
Here,  given  $h\in L$ and $f\in L^{\ast}$, the element $h\otimes f$ is denoted by $h\bowtie f\in D(L)$, 
while $\langle f,h\rangle$ is the evaluation of $f$ in $h$. 

\medbreak
We first deal with $H^*\simeq\toba(V^{\ast})\#\ku^{\Gamma}$. 
Let $w_1,w_2\in H^{\ast}$ be defined by
\begin{align*}
&\begin{aligned}
&w_1(x_1^{m_1}x_{21}^{m_2}x_2^{n}g^{m_3})= \delta_{m_1,0}\delta_{m_2,0}\delta_{m_3,0}\delta_{n,1},\\ &w_2(x_1^{m_1}x_{21}^{m_2}x_2^{n}g^{m_3})=\delta_{m_1,1}\delta_{m_2,0}\delta_{m_3,0}\delta_{n,0},&
\end{aligned}
& m_i &\in \I_{0,1}, & i &\in \I_{3},& n &\in \I_{0,3}.
\end{align*}
\medbreak

Let $\gamma\in \ku^{\Gamma}$, $\gamma(g^{m})=\delta_{m,1}$, $m \in \I_{0,1}$, thus $\ku^{\Gamma} \simeq \ku[\gamma]/(\gamma^2 + \gamma)$. 
By \cite[Lemma 2.1]{ap1}, $V^{\ast} \in \yd{\ku^{\Gamma}}$ with basis $\{w_i: i \in \I_2\}$, action 
and coaction  given respectively by  
\begin{align}\label{struc-yd-dual}
&\gamma \cdot w_i=w_i,& i \in \I_2,& &\delta(w_1)=1\otimes w_1,& &\delta(w_2)=1\otimes w_2+\gamma\otimes w_1.
\end{align} 

\begin{lemma} \label{lem:nichols-dual}
The algebra $\toba(V^{\ast})$ is generated by $w_1,w_2$ with relations
\begin{align} \label{eq-dual-jordan-plane-char2}
&w_1^2,&  &w_2^4,& &w_2^2w_1+w_1w_2^2+w_1w_2w_1,& &w_1w_2w_1w_2+w_2w_1w_2w_1.&
\end{align}
If $w_{21}\coloneqq w_1w_2+w_2w_1 \in\toba(V^{\ast})$, then $\toba(V^{\ast})$ has a PBW-basis 
\[\{w_1^{m_1}w_{21}^{m_2}w_2^{n}\,:\,m_i\in\I_{0,1},\, i \in \I_2, \, n\in \I_{0,3}\}.\]
\end{lemma}
\pf
It is easy to see that $(V^{\ast},c^{\ast}) \simeq (V,c^{-1})$. Then
$\toba(V^{\ast})$ has the same relations of $\toba(V)$ with $w_i$ instead of $x_i$, by \cite[Lemma 1.11]{ahs}.
\epf

\begin{lemma}\label{lem:hopf-dual} (See \cite[Proposition 2.2]{ap1}). The algebra
$H^{\ast}\simeq \toba(V^{\ast})\#\ku^{\Gamma}$ is generated by $w_1,w_2,\gamma$ with relations \eqref{eq-dual-jordan-plane-char2} and
\begin{align}\label{gen-dual-boson}
w_i\gamma=\gamma w_i+w_i,\quad i \in \I_2, & &\gamma^2=\gamma.
\end{align}
The coproduct of $H^{\ast}$ is given by
\begin{align}\label{for-copro-dual-boson}
&w_1, \gamma \in \mathcal P(H^{\ast}), &
&\Delta(w_2)=w_2\otimes 1+1\otimes w_2+\gamma\otimes w_1. \qed
\end{align}
\end{lemma}

Clearly $D(\ku \Gamma)$ is generated by  $g$, $\gamma $ with relations
\begin{align}\label{double:group}
g^2=1,& &\gamma^2=\gamma,& &\gamma g= g\gamma.
\end{align}
We can then present the Drinfeld double  $D(H)$ by generators and relations.

\begin{prop}\label{prop:double}
The algebra $D(H)$ is   generated by $x_1,x_2,g,w_1,w_2,\gamma$ with relations \eqref{eq-jordan-plane-char2}, \eqref{rel-boson}, \eqref{eq-dual-jordan-plane-char2}, \eqref{gen-dual-boson}, \eqref{double:group} and 
\begin{align*}
	& w_1x_1=x_1w_1,& 	& w_2x_1=x_1(w_1+w_2) + 1+g,& &\gamma x_1=x_1\gamma+x_1,&  \\
	& w_1x_2=x_2w_1+1+g,&  & w_2x_2=x_2(w_2+w_1)+g\gamma,& &\gamma x_2=x_2\gamma+x_2,&\\
	& w_1g=gw_1,& & w_2g=g(w_1+w_2).
\end{align*}

The coproduct of $D(H)$ is determined by \eqref{for-copro-boson} and \eqref{for-copro-dual-boson}. The algebra $D(H)$ has dimension $2^{10}$ and has a PBW-basis 
$$\{x_1^{m_1}x_{21}^{m_2} x_2^{n_1}g^{m_3}\gamma^{m_4} w_1^{m_5} w_{21}^{m_6}w_2^{n_2}\,: \, m_i \in \I_{0,1},\, i \in \I_6,\, n_j \in \I_{0,3}, \,\, j \in \I_2\}.$$
\end{prop}
\pf
This follows from \cite[Proposition 2.4]{ap1} and the previous results.
\epf

Here are some properties of $D(H)$. For more details see \cite[Definition 12.3.4]{Ra}.

\begin{prop}\label{ribbon}
The Hopf algebra
$D(H)$ is  pointed and ribbon.
\end{prop}
\pf
Consider the following subspaces of $D(H)$:
\begin{align*} & C_{\{0\}}= \ku\{1,g\}, & &   C_{\{1\}}=  C_{\{0\}} + \ku\{x_1,x_2,w_1,\gamma\}, & &  C_{\{2\}}= C_{\{1\}} + \{w_2\}.
\end{align*}
This is a coalgebra filtration of $C = C_{\{2\}}$, thus  the coradical $C_0$ of $C$ equals $C_{\{0\}}$
by \cite[Proposition 4.1.2]{Ra}.
Since $C$ generates $D(H)$ as algebra,   $D(H)_0\subseteq C_{\{0\}}$ by \cite[Corollary 5.1.12]{Ra}. 
Thus $D(H)$ is pointed.
Now $D(H)$ is unimodular by \cite[Theorem 13.2.1]{Ra}.
Clearly $D(H)$ is quasitriangular and  $\Ss^2(x)=gxg^{-1}$,  $x \in D(H)$. By \cite[Proposition 3]{kr}, $D(H)$ is ribbon.
\epf 

\subsection{\texorpdfstring{$D(H)$}{} as an abelian extension}\label{sub:abelian-ext-A} 
Recall that a pair $A\overset{\iota}\hookrightarrow C \overset{\pi}\twoheadrightarrow B$ 
of  Hopf algebra maps
is a {\it short exact sequence} if $\iota$ is injective, $\pi$ is surjective,  
$\ker \pi=C\iota(A)^{+}$,  and  $\iota(A)=C^{\rm co\,\pi}$ ( see e.g. \cite{ad}).
We also say that $C$ is an extension of $A$ by $B$. When $B$ is cocommutative and $A$ is commutative,
 the extension is called \emph{abelian}.

\begin{remark} \label{exactseq}
If $A \overset{\iota}\hookrightarrow C$ is injective and faithfully flat, $\iota(A)$ is stable by the left adjoint action of $C$,  $\pi$ is surjective and $\ker \pi=C\iota(A)^{+}$, then the sequence $A \overset{\iota}\hookrightarrow C \overset{\pi}\twoheadrightarrow B$ is exact; see \cite[Corollary 1.2.5, 1.2.14]{ad}, \cite{sch}.
\end{remark}

\begin{theorem}\label{teo:sequence}
The subalgebra $\nucleo$  of $D(H)$ generated by $x_1$, $x_{21}$, $w_1$, $w_{21}$ and $g$ 
is a normal local commutative Hopf subalgebra with defining relations 
\begin{align}\label{eq:sequence}
x_1^2=0,& &x_{21}^2=0,& & w_1^2=0,& &w_{21}^2=0,& & g^2=1.
\end{align}
Also, $\dim \nucleo = 2^5$. Let $\pi: D(H) \to \mathfrak{u}(\mgo)$ be the algebra map given by
\begin{align*}
\pi(x_1) =\pi(w_1) &= \pi(g - 1) =0,  &
\pi(x_2)&= a,& \pi(w_2)  &= b, & \pi(\gamma) &= c.
\end{align*}
Then $\pi$ is a map of Hopf algebras, the short sequence $\nucleo\overset{\iota}\hookrightarrow D(H)\overset{\pi}\twoheadrightarrow \mathfrak{u}(\mgo)$ is exact and $D(H)$ is an abelian extension. 
\end{theorem}

\pf
By Proposition \ref{prop:double}, $\nucleo$ is a (local) commutative Hopf subalgebra of $D(H)$ and $\dim \nucleo = 2^5$.  Let $\widetilde{\nucleo}$ be the commutative Hopf algebra
generated by $x_1$, $x_{21}$, $w_1$, $w_{21}$ and $g$  with defining relations
\eqref{eq:sequence}; in particular, the relations $x_{1}x_{21} = x_{21}x_{1}$
and $w_{1}w_{21} = w_{21}w_{1}$ hold in $\widetilde{\nucleo}$ by the commutativity assumption. The coproduct of $\widetilde{\nucleo}$ is given by: $x_1\in \cP{g}{1}{\widetilde{\nucleo}}$, $w_1\in \mathcal{P}(\widetilde{\nucleo})$, $\Delta(x_{21})=x_{21}\otimes 1+1\otimes x_{21}+x_1g\otimes x_1$,  $\Delta(w_{21})=w_{21}\otimes 1+1\otimes w_{21}+w_1\otimes w_1$ and $g$ is a group-like element of $\widetilde{\nucleo}$. Clearly there is a Hopf algebra isomorphism from $\widetilde{\nucleo}$ to $\nucleo$. By direct calculations, we verify that $\nucleo$ is normal. For instance, we have $(\ad_r x_2)(g)=\Ss((x_2)_{(1)})g(x_2)_{(2)}=x_1\in T$ and $(\ad_r x_2)(x_1)=x_{21}g\in T$. Also, the following commutation relations are useful to prove that $T$ is normal: 
$w_2x_{21}=x_{21}w_2+x_1(1+g)$ and $x_2w_{21}=w_{21}x_2$. Hence $D(H)\nucleo^{+}$ is a Hopf ideal of $D(H)$. 
Clearly, $\pi$ is well-defined, preserves the comultiplication and induces an isomoprhism between the quotient $D(H)/D(H)\nucleo^{+}$ and $ \mathfrak{u}(\mgo)$.
\epf

\section{Properties of the algebra \texorpdfstring{$\mathfrak{u}(\mgo)$}{}}\label{sec:tame}
Let $\mgo$ be as in \S \ref{subsection:restricted}. 
Here we prove that the representation type of $\mathfrak{u}(\mgo)$ is tame and that the basic algebra associated to $\mathfrak{u}(\mgo)$ is special biserial. 

\subsection{The simple  \texorpdfstring{$\mathfrak{u}(\mgo)$}{}-modules}\label{sub:irred-rep-u(g)}
Let $V_{0}$, respectively $V_1$, denote the tri\-vi\-al one-dimensional  $\mathfrak{u}(\mgo)$-module,
respectively the three-dimensional $\mathfrak{u}(\mgo)$-module $\mfG$ with the adjoint representation $\ad$, 
cf. Remark \ref{rem:adjoint-representation}.  This is given in the basis $\{v_1, v_2, v_3\} \coloneqq  \{b,c,a\}$ of $\mfG$ by
$\ad a = \Att$, $\ad b = \Btt$, $\ad c = \Ctt$, where
\begin{align}
\label{mod-simple-3} \Att &=\left( \begin{matrix} 0 & 0 & 0 \\ 1 & 0 & 0 \\ 0 & 1 & 0 \end{matrix} \right),&  
\Btt &=\left( \begin{matrix} 0 & 1 & 0 \\ 0 & 0 & 1 \\ 0 & 0 & 0 \end{matrix} \right),&  
\Ctt &=\left( \begin{matrix} 1 & 0 & 0 \\ 0 & 0 & 0 \\ 0 & 0 & 1 \end{matrix} \right).&
\end{align}

The following result follows from \cite{do}; we include a proof for completeness.

\begin{prop}\label{prop:simple modules}
$\Irr \mathfrak{u}(\mgo) = \{V_0, V_1\}$.
\end{prop}
\begin{proof}
Let $V \in \Irr \mathfrak{u}(\mgo) $. Since $b^2c=cb^2$, there exist $0\neq v\in V$ and $\beta,\xi\in \ku$ such that $b^2v=\beta v$
and $cv=\xi v$. From $b^4=0$ follows that $\beta=0$. We can assume that $bv=0$. In fact, if $bv\neq 0$ we replace $v$ by $bv$ and $\xi$ by $\xi+1$.  
Let $W = \ku \{a^iv: i \in \I_{0,3}\}$, a linear subspace of $V$. Since 
\begin{align*}
ca^j&=a^jc+ja^j, & j &\in \N_0,& ba^2 &= a^2b+a & \text{and } ba^3 &= a^3b+a^2(c+1), 
\end{align*}
we have that
\begin{align*}
&c(av)=(\xi+1)av,& &c(a^2v)=\xi a^2v,& &c(a^3v)=(\xi+1)a^3v,\\
&b(av)=\xi v,& &b(a^2v)= av,& &b(a^3v)=(\xi+1)a^2v.
\end{align*}
Hence $W \neq 0$ is a $\mathfrak{u}(\mgo)$-submodule of $V$ and so $W=V$.
We consider various cases.
First, if $av\in \ku\{v\}$, then $V=\ku\{v\}$ and  $V \simeq V_0$. 
\medbreak
If $\{a^iv: i \in \I_{0,1}\}$ is linearly independent and $a^2v\in \ku\{a^iv: i \in \I_{0,1}\}$, then there exist $\lambda_i \in \ku$, $i \in \I_2$, such that $a^2v=\lambda_1v+\lambda_2av$. Applying $b$ to this equality we get $ba^2v=\lambda_2bav$, i.e. $av=\lambda_2\xi v$,
a contradiction. 
\medbreak
Now suppose that $\{a^iv: i \in \I_{0,2}\}$ is linearly independent and $a^3v\in \ku\{a^iv: i \in \I_{0,2}\}$. Thus there exist $\lambda_i\in \ku$, $i \in \I_3$, such that $a^3v  = \lambda_1v+\lambda_2av+\lambda_3a^{2}v$. 
Acting by $b$, we get that  $(\xi+1)a^2v+\lambda_3av+\lambda_2\xi v=0$. 
So $\xi=1$ and $\lambda_2=\lambda_3=0$ and $a^3v  = \lambda_1v$.
Applying $a$, we have $0 = a^4v = \lambda_1av$, thus $\lambda_1 = 0$. Therefore 
$V\simeq V_1$.

\medbreak
Assume finally that $\{a^iv: i \in \I_{0,3}\}$ is a basis of $V$.
Since $c^2+c=0$, we have $\xi=0$ or $1$. 
If $\xi=1$, respectively $\xi =0$, then  $U=\ku\{a^3v\}$, respectively  $U = \ku \{a^iv: i \in \I_{3}\}$, 
is a submodule of $V$, hence $V= U$, a contradiction.
\end{proof}

\begin{remark}From Proposition \ref{prop:simple modules}, or  directly,  we deduce that  $V_1^{\ast} \simeq V_1$.
\end{remark}

By Theorem \ref{teo:sequence}, $\mathfrak{u}(\mgo)\simeq D(H)/I$ where $I$ is a nilpotent ideal. Thus
$\Irr D(H) \simeq \Irr\mathfrak{u}(\mgo)$. We conclude from Proposition \ref{prop:simple modules}:

\begin{coro}\label{cor:simples-d(h)} $\Irr D(H) = \{V_0, V_1\}$ where $V_0$ is the trivial $D(H)$-module 
and $V_1$  is the three-dimensional module where $x_1$, $x_{21}$, $w_1$ and $w_{21}$ act by $0$;
$g$ acts by $\id$; and 
$x_2$, $w_2$,  $\gamma$ act by the matrices $\Att, \Btt, \Ctt$ in \eqref{mod-simple-3}.\qed
\end{coro}

\subsection{Extensions of the simple \texorpdfstring{$\mathfrak{u}(\mgo)$}{}-modules}
Our next goal  is to compute the groups $\ext^1_{\mathfrak{u}(\mgo)}(V_i,V_j)$, $i,j \in \I_{0,1}$.
Recall from Lemma \ref{relbasic} that
\begin{align}\label{eq:a2-b2}
a &= a^2b+ba^2,& b &= b^2a + ab^2.
\end{align}
Hence if $a^2=0$ or $b^2=0$ acting on an $\mathfrak{u}(\mgo)$-module $V$, then  $a=b=c=0$ in $\End V$.

\begin{lemma}\label{lema:ext-triv=triv}
Any  two-dimensional  $\mathfrak{u}(\mathfrak{\mgo})$-module is trivial.
Consequently,  \[\dim \ext^1_{\mathfrak{u}(\mgo)}(V_0,V_0)=0.\]
\end{lemma}

\pf The previous remark applies since $a^2$ acts by $0$ on such a module.
\epf

We shall need the following more general result.

\begin{lemma}\label{block} Let $V \in \lmod{\mathfrak{u}(\mathfrak{\mgo})}$.
Then, $\rho_V(a)$ has no Jordan block of size $2$. 
\end{lemma}

\pf  Let  $\{v_i: i \in \I_m\}$ be a basis of $V$ such that $av_i = \lambda_i v_{i+1}$, where $\lambda_i \in \I_{0,1}$ for $i<m$ and $av_m = 0$.
Then $bv_j=\sum\limits_{k\in \I_m} \beta_{jk}v_k$, where $\beta_{jk}\in \ku$.
Pick $i \in \I_m$ such that $av_i=v_{i+1}$ and $av_{i+1}=0$.
By \eqref{eq:a2-b2} we have
\[ v_{i+1} = av_{i} = a^2bv_{i} = a^2 \sum\limits_{k\in \I_m} \beta_{ik}v_k
=  \sum\limits_{k\in \I_m} \beta_{ik} \lambda_k \lambda_{k+1}v_{k+2}.
\]
If $i=1$, then  $v_2 \in \ku \{v_i: i \in \I_{3,m}\}$, a contradiction.
If $i >1$, then
\[ 1 =  \beta_{i i-1} \lambda_{i-1} \lambda_{i}.
\]
Thus $\lambda_{i-1}  \neq 0$, i.~e.,  $av_{i-1} = v_{i}$ and the claim follows.
\epf 

For any $\vartheta,\lambda, \mu \in \ku$, $V_{\vartheta, \lambda,\mu}\in \lmod{\mathfrak{u}(\mgo)}$ denote $ \ku^4$   
with action given by 
\begin{align*}
a\mapsto\left( \begin{matrix} 0 & 0 & 0 & 0 \\ \vartheta & 0 & 0 & 0\\ 0 & 1 & 0 & 0\\ 0 & 0 & 1 & 0 \end{matrix} \right),
& & b \mapsto \left( \begin{matrix} 0 & 0 & 0 & 0\\ 0 & 0 & 1 &0\\ \lambda & 0 & 0& 1\\ \mu & 0 & 0 & 0 \end{matrix} \right),
&   &c \mapsto \left( \begin{matrix} 0 & 0 & 0 & 0\\ 0 & 1 & 0& 0 \\ 0 & 0 & 0 & 0\\  \lambda & 0 & 0 & 1 \end{matrix} \right).&
\end{align*}

\begin{lemma}\label{lema:ext-v0-v1}
If $U$ is an extension of $V_0$ by $V_1$, then $U$ is isomorphic to $V_{\vartheta,\lambda, \mu}$, for some
$\vartheta,\lambda, \mu \in \ku$. Also, 
\begin{align} \label{eq:ext-v0-v1}
\dim \ext^1_{\mathfrak{u}(\mgo)}(V_0,V_1)=2.
\end{align}
\end{lemma}

\pf Let $0\rightarrow {V_1}\overset{\varphi}\rightarrow U\overset{\psi}\rightarrow V_0 \rightarrow 0$ 
be an exact sequence in $\lmod{\,\mathfrak{u}(\mgo)}$. We identify $V_1$ with a subspace of $U$. 
Recall the basis $\{v_i: i \in \I_3\}$ of $V_1$ implementing the action of $\mathfrak{u}(\mgo)$ by \eqref{mod-simple-3}. 
Fix $w_0 \in V_0 \backslash 0$.
Let $v_0\in U$ such that $\psi(v_0) = w_0$; then  $\{v_i: i \in \I_{0,3}\}$ is a basis of $U$ and 
$bv_0 \in V_1$. Thus there are  $\kappa,\lambda, \mu\in \ku$ such that 
\[bv_0 = \kappa v_1+ \lambda v_2 + \mu v_3.\]
Let $m_a$ denote the minimal polynomial of $a_{\vert U}$; then $m_a = X^3$ or $X^4$. 
\begin{step} There exists $u\in U$ such that $au = v_1$ (i.e., $m_a = X^4$). 
\end{step}
Let  $\vartheta\in \kut$ be such that $\psi(u) = \vartheta^{-1} w_0$; set $v_0 = \vartheta u$ so that $\psi(v_0) = w_0$
and $a v_0 = \vartheta v_1$.
Now $cv_0 = (ab+ba)v_0 = \kappa v_2 + \lambda v_3 $; hence $\kappa =0$ because 
\begin{align*}
v_1 &= \vartheta^{-1} av_0 = \vartheta^{-1} (ac+ca)v_0 = \vartheta^{-1} \kappa v_3 + v_1. 
\end{align*}

\begin{step} There exists $u\in U \backslash V_1$ such that $au= 0$ (i.e. $m_a = X^3$). 
\end{step}

Let $v_0 \in \ku \{u\}$ be such that $\psi(v_0) = w_0$. Then $cv_0 = \kappa v_2 + \lambda v_3 $ and
$0 = av_0 = (ac+ca)v_0 = a(\kappa v_2 + \lambda v_3)  = \kappa v_3$; thus $\kappa = 0$.
Here we set $\vartheta = 0$.

\medbreak
In both cases $U\simeq V_{\vartheta, \lambda,\mu}$ by identifying $\{v_i: i \in \I_{0,3}\}$ to the canonical  basis.
Assume next that there exists an isomorphism $T: V_{\vartheta, \lambda,\mu} \to 
V_{\widetilde{\vartheta}, \widetilde{\lambda},\widetilde{\mu}}$ such that the following diagram commutes:
\begin{align*}
\xymatrix{ 0  \ar@{->}[r] & V_1 \ar@{->}[r]  \ar@{=}[d] 
& V_{\vartheta, \lambda,\mu} \ar@{->}[r]  \ar@{->}^{T}[d] &V_0 \ar@{->}[r]  \ar@{=}[d] \ar@{->}[r]  &0\\
0  \ar@{->}[r] & V_1 \ar@{->}[r]  
& V_{\widetilde{\vartheta}, \widetilde{\lambda},\widetilde{\mu}} \ar@{->}[r]  
&V_0 \ar@{->}[r]  \ar@{->}[r]  &0.}
\end{align*} 
Then $T(v_i) = v_i$, $i\in \I_3$, while $T(v_0) = v_0 +  \nu v_1+ \xi v_2 + \zeta v_3 $.
Since $T$ preserves the action, we get $\nu = \xi = 0$,  $\vartheta = \widetilde{\vartheta}$,
$ \lambda = \widetilde{\lambda} + \zeta$, $\mu = \widetilde{\mu}$; this implies \eqref{eq:ext-v0-v1}.
\epf 

Clearly  $V_{\vartheta, \lambda,\mu}$ is indecomposable unless $(\vartheta, \lambda,\mu) = 0$.

Let $L$  be a Hopf algebra and $M\in \lmod{L}$. Let  $M^*, {}^*M \in \lmod{L}$  denote
the linear dual of $M$ with the action given by $x\cdot \alpha = \alpha \circ \Ss(x)$, respectively 
$x\cdot \alpha = \alpha \circ \Ss^{-1}(x)$, for $x\in L$, $\alpha \in \Hom(M, \ku)$. 
Since the contravariant functors $M \mapsto M^*$ and $M \mapsto {}^* M$ are exact, we have natural isomorphisms
\begin{align}
\ext^1_{L}(M, N) &\simeq \ext^1_{L}(N^*, M^*)  \simeq \ext^1_{L}({}^*N, {}^*M) ,& N, M \in \lmod{L}.
\end{align}

\begin{lemma}\label{lema:ext-v1-v0}
If $U$ is an extension of $V_0$ by $V_1$, then $U\simeq V^*_{\vartheta,\lambda, \mu}$, for some
$\vartheta,\lambda, \mu \in \ku$.  Also,
$\dim \ext^1_{\mathfrak{u}(\mgo)}(V_1,V_0)=2$. \qed
\end{lemma}

\begin{lemma} $\ext^1_{\mathfrak{u}(\mgo)}(V_1,V_1)=0$. 
\end{lemma}
\pf 
Let $0\rightarrow {V_1}\overset{\varphi}\rightarrow U\overset{\psi}\rightarrow V_1 \rightarrow 0$ be an exact sequence
in $\lmod{\,\mathfrak{u}(\mgo)}$; identify $V_1$ with a subspace of $U$. 
Recall the basis $\{v_i: i \in \I_3\}$ of $V_1$ implementing the action of $\mathfrak{u}(\mgo)$ by \eqref{mod-simple-3}. 
Pick $v_4\in U$ such that  $\psi(v_{4}) = v_1$ and set $v_5 = av_4$, $v_6 = av_5$. Thus 
$\psi(v_{5}) = v_2$ and $\psi(v_{6}) = v_3$. Hence  $\{v_i: i \in \I_6\}$ is a basis of $U$. 
Since $av_3 = 0$ and $a^2 v_6 =  a^4v_4= 0$, then  $av_6 \in V_1$ and $av_6  = \lambda v_3$ for some $\lambda\in \ku$.
If $\lambda \neq 0$, then $\ku\{v_1 + \lambda^{-1}v_5, v_2 + \lambda^{-1}v_6\}$ provides a Jordan block of size $2$ for
the action of $a$, contradicting Lemma \ref{block}. Thus $av_6  = 0$.
Since $\psi$ is a morphism, there are  $\vartheta_i,\lambda_i,\mu_i \in \ku$, $i \in \I_{3}$, such that 
\begin{align*}
& bv_4=\sum_{i \in \I_3} \vartheta_i v_i, & & bv_5= v_4+ \sum_{i \in \I_3} \lambda_i v_i, & & bv_6= v_5+\sum_{i \in \I_3} \mu_i v_i.
\end{align*}

By the relation $c=ab+ba$, we have that 
\begin{align*}
&cv_4 =  \lambda_1v_1+(\vartheta_1+\lambda_2)v_2+(\vartheta_2+\lambda_3)v_3+v_4,\\
&cv_5=\mu_1v_1+(\lambda_1+\mu_2)v_2+(\mu_3+\lambda_2)v_3,\\
&cv_6 = v_6 + \mu_1v_2+\mu_2v_3.
\end{align*}
From $a=ac+ca$, we obtain that $\lambda_1=\mu_1=0$ and $\mu_3=\vartheta_1$. 
Also, since $b=bc+cb$, we get that $\mu_2 =\lambda_2 =\vartheta_3=0$, $\vartheta_2 = \lambda_3$. 
We conclude that  $U = V_1 \oplus Z$, where $Z \coloneqq  \ku\{v_4+\vartheta_1 v_2+ \vartheta_2 v_3, v_5+\vartheta_1v_3, v_6\}$,  
is an $\mathfrak{u}(\mgo)$-submodule.
\epf

The {\it Gabriel quiver} of $\mathfrak{u}(\mgo)$ is the quiver $\ext Q(\mathfrak{u}(\mgo))$ with vertices $\{0,1\}$ 
and $\dim \ext^{1}_{\mathfrak{u}(\mgo)}(S_i,S_j)$ arrows from the vertex $i$ to the vertex $j$, that is

\begin{align}\label{eq:quiver}  \xymatrix{0\ar@/^1pc/[rr] \ar@/^2pc/[rr]&& 1 . \ar@/^1pc/[ll]\ar@/^2pc/[ll]} \\ \notag
\end{align}

Thus the separated quiver of $\ext Q(\mathfrak{u}(\mgo))$ is a disjoint union of two Dynkin diagrams of type $B_2$. 
See \cite[X.2.6]{ars}. In the next subsection we will show that the representation type of  $\mathfrak{u}(\mgo)$
is actually tame.

\subsection{Projective covers of simple modules}
Let $V_0$ and $V_1$ be the simple $\mathfrak{u}(\mgo)$-modules as in $\S\,$\ref{sub:irred-rep-u(g)}. 
Let $P(V_i)$ and $I(V_i)$ denote the projective cover and the injective envelope of $V_i$, $i\in \I_{0,1}$, 
respectively. By \cite[Corollary 8.4.3]{Ra}, $\mathfrak{u}(\mgo)$ is a Frobenius algebra; hence  
it is self-injective and $P(V_i) \simeq I(V_i)$, $i \in \I_{0,1}$. In order to determine $P(V_0)$, 
we define the eight-dimensional $\mathfrak{u}(\mgo)$-module $M$ with basis $\{v_i,w_i\,:\,i\in \I_{4}\}$ 
and action given by
\begin{align*}
&av_i=v_{i+1},\,\, i\in\I_{3},& & bv_i=0,\,\, i=1,4, & & c{v_i}=v_i,\,\, i=1,3,  \\
&av_4=0, & &bv_{i}=v_{i-1},\,\,  i\in \I_{2,3}, & &cv_i=0,\,\, i=2,4, \\
&aw_i=w_{i+1},\,\, i\in\I_{3},& & bw_i=v_{i+2},\,\, i\in \I_{2}, & & cw_i=0,\,\, i=1,3,\\
&aw_4=0,& & bw_{i}=w_{i-1},\,\, i \in \I_{3,4}, & & c{w_{i}}=w_{i},\,\, i=2,4. 
\end{align*}
We can see the action of $\mathfrak{u}(\mgo)$ on $M$ by a directed graph. The arrows 
oriented from left to right, indicate the action of $a$ while the 
arrows oriented from right to left represent the action of $b$. Also, $\circ_{u}$ means $cu=0$ while $\bullet_{u}$ means $cu=u$, $u\in M$. Then the directed graph associated to $M$ is
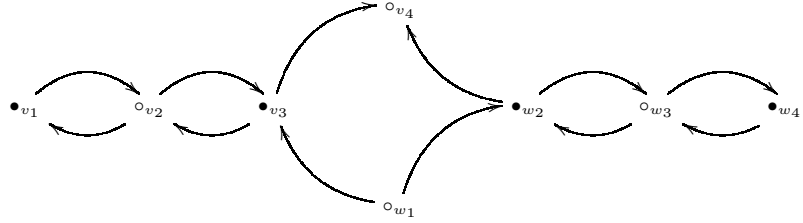
\begin{align}\label{graph-1}
\begin{aligned} 
{\tiny\xymatrix@C=10mm@R=9mm{   
& & & &\circ_{v_4}&&& \\
&\bullet_{v_1} \ar@/^1pc/[r]
&\circ_{ v_2} \ar@/^1pc/[r] \ar@/^1pc/[l]
&\bullet_{v_3} \ar@/^1pc/[l] \ar@{->}@/^1pc/[ur]
& &\bullet_{w_2}\ar@/^1pc/[r] \ar@{->}@/^1pc/[ul]
&\circ_{w_3}\ar@/^1pc/[r]\ar@/^1pc/[l]
&\bullet_{w_4}.\ar@/^1pc/[l]\\
& & & &\circ_{w_1}\ar@{->}@/^1pc/[ur] \ar@{->}@/^1pc/[ul]&&&}}\end{aligned} \end{align}

It is clear that the vector subspaces $M_0=\Bbbk\left\{ v_4\right\}$, $M_1=\Bbbk\left\{ v_i:\,i\in \I_{4}\right\}$ and $M_2\coloneqq M_1\oplus\Bbbk\left\{w_i: \,i\in \I_{2,4} \right\}$ 
are $\mathfrak{u}(\mgo)$-submodules of $M$.

\begin{lemma}\label{comp-series-M} Keep the notation above.
\begin{enumerate}[leftmargin=*,label=\rm{(\roman*)}]
\item $M_0\subseteq M_1\subseteq M_2\subseteq M$ is a composition series of $M$ with  factors 
\begin{align*}
&M_0\simeq V_0,& &M_1/M_0\simeq V_1,&&M_2/M_1\simeq V_1,& &M/M_2\simeq V_0.& 
\end{align*}
\item $\soc M=V_0$. In particular, $M$ is an indecomposable $\mathfrak{u}(\mgo)$-module.\vspace{.1cm}
\item  $\rad M=M_2$. \vspace{.1cm}
\item $\dim P(V_0)\geq \dim M=8$.
\end{enumerate} 
\end{lemma} 

\begin{proof}
The items (i), (ii) and (iii) are obtained through direct calculations.  For (iv), consider $\vartheta:V_0\to M$  the inclusion given by (i)
and let $\iota: V_0\to I(V_0)$ be the canonical  inclusion. 
Since $I(V_0)$ is the injective envelope of $V_0$,  there exists a morphism $\kappa:M\to I(V_0)$ 
such that $\kappa \vartheta=\iota$. 
Let $K=\ker \kappa$. 
If  $K \neq 0$, then $0\neq \soc K\subseteq \soc M=\vartheta (V_0)$. 
But $\kappa \vartheta=\iota$, hence $K\cap \vartheta (V_0)=0$, which is a contradiction. 
Therefore, $\kappa$ is injective and the result follows.
\end{proof}

Next let $N\in \lmod{\mathfrak{u}(\mgo)}$ with basis $\{v_i,w_i\,:\,i\in \I_{4}\}$  and action given by 
\begin{align*}
&av_i=v_{i+1},\,\, i\in\I_{3},& & bv_i=0,\,\, i\in \I_2, & & c{v_i}=0,\,\, i=1,3,  \\
&av_4=0, & &bv_{i}=v_{i-1},\,\,  i\in \I_{3,4}, & &cv_i=v_i,\,\, i=2,4, \\
&aw_i=w_{i+1},\,\, i\in\I_{3},& & bw_i=v_{i},\,\, i=1,4, & & cw_i=w_i,\,\, i=1,3,\\
&aw_4=0,& & bw_{i}=v_i+w_{i-1},\,\, i \in \I_{2,3}, & & c{w_{i}}=0,\,\, i=2,4. 
\end{align*}
As in (\ref{graph-1}), we can see the action of $\mathfrak{u}(\mgo)$ on the module $N$ by the directed graph \vspace{.5cm}
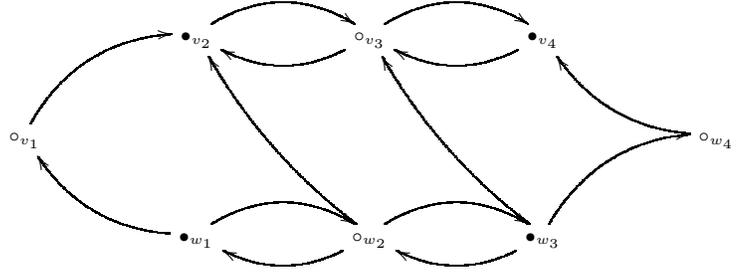
\begin{align}
\begin{aligned} \label{graph}
{\tiny\xymatrix@C=16mm@R=9mm{
& &\bullet_{v_2} \ar@/^1pc/[r]
&\circ_{v_3} \ar@/^1pc/[r] \ar@/^1pc/[l]
&\bullet_{ v_4} \ar@/^1pc/[l]
& \\
&\circ_{v_1}\ar@{->}@/^1pc/[ur]
&
&
&
&\circ_{w_4}.\ar@{->}@/^1pc/[ul]\\
&&\bullet_{w_1}\ar@/^1pc/[r] \ar@{->}@/^1pc/[ul]
&\circ_{w_2}\ar@/^1pc/[r]\ar@/^1pc/[l]\ar@{->}@/^/[luu]
&\bullet_{w_3}\ar@/^1pc/[l]\ar@{->}@/^/[luu]\ar@{->}@/^1pc/[ur]
&\\
&}}\end{aligned} \end{align} 
Now  $N_0= \Bbbk\left\{v_i:i\in \I_{2,4}\right\}$, $N_1=N_0\oplus \Bbbk\left\{v_1 \right\}$ and $N_2=N_1\oplus\Bbbk\left\{w_4\right\}$ are submodules of $N$.
Similar to Lemma \ref{comp-series-M}, we have the following result.

\begin{lemma}\label{comp-series-N} Keep the notation above.
\begin{enumerate}[leftmargin=*,label=\rm{(\roman*)}]
\item $N_0\subseteq N_1\subseteq N_2\subseteq N$ is a composition series of $N$ with factors 
\begin{align*}
&N_0\simeq V_1,& &N_1/N_0\simeq V_0,&&N_2/N_1\simeq V_0,& &N/N_2\simeq V_1.& 
\end{align*} 
\item $\soc N=V_1$. In particular, $N$ is an indecomposable $\mathfrak{u}(\mgo)$-module.\vspace{.1cm}
\item $\rad N=N_2$. \vspace{.1cm}
\item $\dim P(V_0)\geq \dim N=8$. \qed
\end{enumerate} 
\end{lemma}

\begin{prop} \label{proje-cover} Let $V_0$ and $V_1$ be the simple $\mathfrak{u}(\mgo)$-modules. Then,
\begin{align*}
&P(V_0)\simeq I(V_0) \simeq M,&&P(V_1) \simeq I(V_1) \simeq N.&
\end{align*}
\end{prop}
\begin{proof}
Since $\mathfrak{u}(\mgo)\simeq P(V_0)\oplus P(V_1)^3$ as $\mathfrak{u}(\mgo)$-modules, we see that 
\[32=\dim \mathfrak{u}(\mgo)=\dim P(V_0)+3\dim P(V_1).\] 
Then, by Lemmas \ref{comp-series-M} and \ref{comp-series-N} we have that $\dim P(V_i)=8$, $i \in \I_{0,1}$.
\end{proof}

Given  an algebra $A$, $\dim A < \infty$,  $M$ in $\lmod{A}$ is {\it uniserial} if it has a unique composition series;
for example $M$ is uniserial when $\soc M$ and $ M/ \soc M $ are simple.
Now $A$ is {\it biserial} if every indecomposable projective  $P$ in $\lmod{A}$
has  uniserial submodules $U$ and $T$ such that $U+T=\rad P$ and $U\cap T$ is either $0$ or simple. 
Biserial algebras have tame representation type  \cite{cb}. 

\begin{lemma} \label{biserial} The algebra $\mathfrak{u}(\mgo)$ is biserial and has tame representation type.
\end{lemma}
\begin{proof}
By Lemma \ref{comp-series-M} (iii), $M_2 =\rad M =U+V$, where $U = M_1$ and $V=\ku\{v_4,w_i: i \in \I_{2,4}\}$. Now $U\cap V=\ku\{v_4\}\simeq V_0$,
hence  $U$ and $V$ are uniserial. Similarly,  $\rad N=N_2 =U'+V'$, where $U' = N_1$ and $V'=\ku\{v_i,w_4: i \in \I_{2,4}\}$, by Lemma \ref{comp-series-N} (iii). Now  $U'\cap V'=\ku\{v_i: i \in \I_{2,4}\}\simeq V_1$ and $U',V'$ are uniserial. 
Since $M \simeq P(V_0)$ and $N \simeq P(V_1)$ are the unique, up to isomorphism, indecomposable projective $\mathfrak{u}(\mgo)$-modules,
$\mathfrak{u}(\mgo)$ is biserial.
\end{proof}

\subsection{The associated bound quiver}\label{subsec:bound-quiver} 
We collect some facts on the representation theory of  a finite-dimensional algebra $A$.
Let $\{e_i: i \in \I_n\}$ be a complete set  of primitive orthogonal idempotents of $A$.
Consider the equivalence relation on $\I_n$ given by $i\sim j \iff Ae_i \simeq Ae_j$ and pick a full set $\J$ of representatives of $\sim$. 
The basic algebra $A^{\mathtt b}$ (associated to the previous choice) is 
\begin{align*}
A^{\mathtt b} & \coloneqq e_A A e_A , &&\text{where} & e_A & \coloneqq \sum_{j\in \J} e_j.
\end{align*}
Different choices of the primitive orthogonal idempotents or of the representatives give rise to different, 
but isomorphic, basic algebras, cf. \cite[I.6.5]{ass}.

\begin{lemma}\label{lema:equiv-categories} The $\ku$-linear functors
\begin{align} \label{functor-te}
\begin{aligned}
\Ind^l_{e_A} & \colon \lmod{A^{\mathtt b}} \to \lmod{A}, & &&
\Ind^l_{e_A}(N) &\coloneqq   Ae_{A} \otimes_{A^{\mathtt b}} N ,
\\ 
\res^l_{e_A} &\colon \lmod{A} \to \lmod{A^{\mathtt b}}, &&&
\res^l_{e_A}(M) &\coloneqq e_{A}M,
\end{aligned}
\end{align}
are equivalences of categories quasi-inverse to each other. 
Hence there are bijections between the sets of  isomorphism classes of  simple, respectively 
indecomposable, objects  in $\lmod{A}$ and  $\lmod{A^{\mathtt{b}}}$.
\end{lemma}
\begin{proof}
It was showed in \cite[I.6.10]{ass} that the functors 
\begin{align*} 
\begin{aligned}
\Ind^r_{e_A} & \colon \rmod{A^{\mathtt b}} \to \rmod{A}, & &&
\Ind^r_{e_A}(N) &\coloneqq   N\otimes_{A^{\mathtt b}}e_AA,
\\ 
\res^r_{e_A} &\colon \rmod{A} \to \rmod{A^{\mathtt b}}, &&&
\res^r_{e_A}(M) &\coloneqq Me_{A},
\end{aligned}
\end{align*}
determine an equivalence of categories. Using the same arguments we obtain the left version and the result follows.
\end{proof}
Recall that a bound quiver is a pair $(Q, I)$ where  $Q=(Q_0,Q_1,s,t)$ is a  quiver and 
$I$ is an admissible ideal in the path algebra  $\ku Q$; here \emph{admisible} means that there is $m \geq 2$ such that 
$R_Q^m \subseteq I \subseteq R_Q^2$, where  $R_Q^m$ is the two-sided of $\ku Q$ generated by the paths of length $m$.
The associated bound quiver algebra is $\ku Q/I$ which is basic \cite[II.2.10]{ass}.
Let $\rep(Q,I)$ be the category of  finite-dimensional (right) representations of $Q$ bound by $I$ as in \cite[III.1.4]{ass}.
By \cite[III.1.6]{ass}, there is an equivalence of categories $\rep(Q,I) \to  \rmod{\ku Q/I}$.

\medbreak
Given a finite-dimensional basic algebra $B$, fix a complete set  of pri\-mi\-tive orthogonal idempotents $\{e_i: i \in \I_n\}$. 
Then the bound quiver $Q_B$ has set of points $\I_n$ and the arrows from $i$ to $j$  are indexed by a basis of
$e_i\left(\Jac B/\Jac^2 B\right) e_j$; this quiver  is isomorphic to  the Gabriel quiver of $B$ by \cite[III.1.14]{ars}
but we will need later the former presentation. Now $B$ is isomorphic to 
the bound quiver algebra  $\ku Q_B/I$ where $I$ is an adequate admissible ideal; see \cite[II.3.7]{ass}.

\begin{definition}\label{def:special-biserial} \cite{skwa}
Let  $(Q, I)$ be a bound quiver.
The algebra $A = \ku Q/I$ is \emph{special biserial} provided that  for any vertex  $i\in Q_0$
\begin{align*}
|\{\alpha\in Q_1\,:\,s(\alpha)=i\}| &\leq 2, &  |\{\alpha\in Q_1\,:\,t(\alpha)=i\}| &\leq 2 ;
\end{align*}
also, given $\alpha, \beta, \gamma\in Q_1$, $\alpha\neq \beta$ one requires that
\begin{align*}
\text{if } s(\alpha)=s(\beta)=t(\gamma), \text{ then }  \gamma \alpha\in I \text{ or }  \gamma\beta \in I,
\\
\text{if } t(\alpha)= t(\beta)=s(\gamma), \text{ then }  \alpha\gamma \in I \text{ or } \beta \gamma \in I.
\end{align*}
Initially such  algebras  were named special in \cite{skwa}, where it was shown that they are biserial.
The name `special biserial' appears already in \cite{HL}.
\end{definition}

\subsection{The algebra \texorpdfstring{$\mathfrak{u}(\mgo)^{\mathtt b}$}{} is special biserial}\label{sub:special-bis}
Since $\mathfrak{u}(\mgo)\simeq \mathfrak{u}(\mgo)^{\text{opp}}$ by the antipode, there is an equivalence of categories
$\lmod{\mathfrak{u}(\mgo)} \simeq \rmod{\mathfrak{u}(\mgo)}$. However we shall not use the antipode; instead,
we calculate in this Subsection one basic algebra 
$\mathfrak{u}(\mgo)^{\mathtt b}$, the bound quiver $(Q, I)$ associated to $\mathfrak{u}(\mgo)^{\mathtt b}$ and
an algebra isomorphism 
$\psi: \mathfrak{u}(\mgo)^{\mathtt b} \to \left(\mathfrak{u}(\mgo)^{\mathtt b}\right)^{\operatorname{opp}}$. 
We also show that $\mathfrak{u}(\mgo)^{\mathtt b}$
is  special   biserial; this will allow us
to compute  $\Indec\mathfrak{u}(\mgo)$ in Section \ref{sec:indec}.

\medbreak
We start by computing primitive  idempotents $e_0$ and $e_1$ in $\mathfrak{u}(\mgo)$ 
generating submodules isomorphic to the projective covers $P(V_0)$ and $P(V_1)$.

\begin{lemma}\label{idempotents}
The elements $e_0=(1+ab+a^2b^2)(1+c)$ and $e_1=(1+a^2b^2)c$ are idempotents in $\mathfrak{u}(\mgo)$ 
and satisfy $P(V_0)\simeq \mathfrak{u}(\mgo)e_0$,
$P(V_1)\simeq \mathfrak{u}(\mgo)e_1$. 
\end{lemma}
\pf
From Lemma \ref{relbasic} we infer that  $e_0$ and $e_1$ are idempotents and satisfy
\begin{align*}
& ab^3e_0=b^2 e_0, & & a^2b^3e_0=be_0, & & a^3b^3e_0= (ab+c)e_0,\\
& ba^3e_0=a^2e_0, & & b^2a^3e_0=ae_0, & & b^3a^3 e_0=(ab+c)e_0.
\end{align*}
Then,  $P(V_0)\simeq \mathfrak{u}(\mgo)e_0$  by the assignment $v_1 \mapsto b^3e_0$ and $w_1 \mapsto e_0$. 
Moreover, 
\begin{align*}
& ba^3e_1=a^3be_1, & & b^2a^3e_1= a^2be_1, & & b^3a^3e_1=abe_1,\\
& ba^2e_1=bae_1+ba^2be_1, & & bae_1=b e_1+babe_1.
\end{align*}
Then,  $P(V_1)\simeq \mathfrak{u}(\mgo)e_1$  by the assignment 
$v_1 \mapsto be_1$ and $w_1 \mapsto e_1$.
\epf 

Let $e=e_0+e_1$. It follows from  Lemma \ref{relbasic} that
the following  relations hold in $\mathfrak{u}(\mgo)$:
\begin{align*}
& e_0a=ae_1+ac, & & e_0a^2=a^2e_0+a^2(1+c), & & e_0a^3=a^3e_1,\\
& e_1a=ae_0, & & e_1a^2=a^2e_1+a^2c, & & e_1a^3=a^3e_0+a^3(1+c),\\
& e_0b=be_1, & & e_0b^2=b^2e_0+b^2(1+c), & & e_0b^3=b^3c,\\
&e_1b=be_0+b(1+c), & & e_1b^2=b^2c, & & e_1b^3=b^3e_0, \\
&ea=ae+ac, & & ea^2=a^2e+a^2, & & ea^3=a^3e+a^3(1+c),\\
&eb=be+b(1+c), & & eb^2=b^2e+b^2, & & eb^3=b^3e+b^3c,\\
&&& ec=ce=e_1.
\end{align*}

Now, set 
\[\mathfrak{u}(\mgo)^{\mathtt b}= e \mathfrak{u}(\mgo) e.\]
We see from the previous relations that
the  basic algebra $\mathfrak{u}(\mgo)^{\mathtt b}$ has a basis 
\begin{align}\label{eq:base-um-basica}
\{e_0, e_1, ae_0, a^3e_1, b^3e_0, be_1, a^3b^3e_0, abe_1\}.
\end{align}
The simple (left) $\mathfrak{u}(\mgo)^{\mathtt b}$-modules are $eV_0 =V_0$ and $eV_1 = \ku \{v_1\}$, see Subsection \ref{sub:irred-rep-u(g)}. Using this,
we observe that the Jacobson radical of $\mathfrak{u}(\mgo)^{\mathtt b}$ is
\begin{align*}
\Jac \mathfrak{u}(\mgo)^{\mathtt b} &= \langle ae_0, a^3e_1, b^3e_0, be_1, a^3b^3e_0, abe_1 \rangle.
\end{align*}
Hence $\Jac^2 \mathfrak{u}(\mgo)^{\mathtt b}  = \langle a^3b^3e_0, abe_1\rangle$ and
\begin{align*}
e_i \left(\Jac \mathfrak{u}(\mgo)^{\mathtt b}/\Jac^2 \mathfrak{u}(\mgo)^{\mathtt b}\right) e_i&=0,&  i &\in \I_{0,1},
\\  e_0 \left(\Jac \mathfrak{u}(\mgo)^{\mathtt b}/\Jac^2 \mathfrak{u}(\mgo)^{\mathtt b}\right) e_1 
&= \ku \{a^3e_1, be_1\}, 
\\  e_1 \left(\Jac \mathfrak{u}(\mgo)^{\mathtt b}/\Jac^2 \mathfrak{u}(\mgo)^{\mathtt b}\right) e_0&=
\ku \{ae_0, b^3e_0\}.
\end{align*}
Therefore the ordinary quiver of $\mathfrak{u}(\mgo)^{\mathtt b}$, denoted by $Q\coloneqq Q_{\mathfrak{u}(\mgo)^{\mathtt b}}$, is 
\begin{align}\label{quiver}  \xymatrix{0\ar@/^1pc/[rr]^{\alpha_2} \ar@/^2pc/[rr]^{\alpha_1} && 1 . \ar@/^1pc/[ll]^{\beta_1} \ar@/^2pc/[ll]^{\beta_2} } \end{align} 
As we said, this coincides with the Gabriel quiver \eqref{eq:quiver}, see \cite[III.1.14]{ars}.

\begin{theorem}\label{teo-iso-quiver} \begin{enumerate*} [leftmargin=*,label=\rm{(\roman*)}]
\item\label{item:bdquiv} Let $Q$ be the quiver \eqref{quiver}. Then, $\mathfrak{u}(\mgo)^{\mathtt b} \simeq \ku Q/I$, where 
\end{enumerate*} 
\begin{align}\label{eq:bound}
I=\langle \alpha_1\beta_1, \alpha_2\beta_2, \beta_1\alpha_1, \beta_2 \alpha_2, \alpha_1\beta_2+\alpha_2\beta_1, \beta_1\alpha_2+\beta_2\alpha_1 \rangle
\end{align} 
is the kernel of the algebra epimorphism  $\varphi: \ku Q \to \mathfrak{u}(\mgo)^{\mathtt b}$ defined by
\begin{align}\label{morphism-fi}
\varphi(\alpha_1) &= a^3e_1,& \varphi(\alpha_2) &= be_1, &\varphi(\beta_1) &=ae_0,& \varphi(\beta_2) &= b^3e_0,
\end{align}
and $\varphi(\varepsilon_i) = e_i$. Here $\varepsilon_i$ is the trivial path of length $0$ associated to $i\in \I_{0,1}$.

\begin{enumerate} [leftmargin=*,label=\rm{(\roman*)}]\setcounter{enumi}{1}
\item\label{item:specbiserial}
The algebra $\mathfrak{u}(\mgo)^{\mathtt b}$ is special biserial.
\end{enumerate} 
\end{theorem}
\begin{proof}
Arguing as in the proof of \cite[II.3.7]{ass}, we see that \eqref{morphism-fi} determines a surjective algebra map $\varphi: \ku Q\to \mathfrak{u}(\mgo)^{\mathtt b}$. In order to prove that $\varphi$ is injective, we denote  $\overline{u}:=u+I\in \ku Q/I$. Observe that $\{\overline{e}_0,\overline{e}_1,\overline{\alpha}_1, \overline{\alpha}_2,\overline{\beta}_1,\overline{\beta}_2, \overline{\alpha_1\beta}_2,\overline{\beta_1\alpha}_2\}$ is a basis of $\ku Q/I$ and hence $\dim \big(\ku Q/I\big)=8$. Since $I\subset \ker \varphi$ and $\dim \big(\ku Q/\ker \varphi\big)=\dim \mathfrak{u}(\mgo)^{\mathtt b}=8$, it follows that $I=\ker \varphi$.
This gives \ref{item:bdquiv} and  then  \ref{item:specbiserial} follows at once by the Definition \ref{def:special-biserial}.
\end{proof}

Finally, let $\psi:\mathfrak{u}(\mgo)^{\mathtt b}\to \mathfrak{u}(\mgo)^{\mathtt b}$  be given by
\begin{align*}
&\psi(e_0)=e_0,& &\psi(e_1)=e_1,&    &\psi(abe_1)=abe_1,& &\psi(a^3b^3e_0)=a^3b^3e_0,\\
&\psi(ae_0)=a^3e_1,& &\psi(a^3e_1)=ae_0,& &\psi(b^3e_0)=be_1,& &\psi(be_1)=b^3e_0.
\end{align*}
A direct verification shows that this is an anti-isomorphism of algebras.

\section{Indecomposable \texorpdfstring{$\mathfrak{u}(\mgo)$}{}-modules} \label{sec:indec}

\subsection{The classification}
Let $(Q, I)$ be a bound quiver such that $A \coloneqq \ku Q/I$ is special biserial. 
The   indecomposable objects in $\rmod{A}$ were determined in \cite[Proposition 2.3]{WW}; we 
follow the approach in \cite[\S 1]{HL}. 
Namely, $\Indec A$ consists of either simple, or projective-injective indecomposable, 
or  string (see Subsection \ref{sub-string}), or  band modules
(see Subsection \ref{sub-band}). Our main result in this Section is:

\begin{theorem}\label{thm:clasif-indec-um} The set $\Indec\mathfrak{u}(\mgo)$ consists of
\begin{enumerate}[leftmargin=*,label=\rm{(\roman*)}]
\item the simple modules $V_i$ and their projective covers $P(V_i)$, $i\in \I_{0,1}$,

\medbreak
\item\label{item:clasif-indec-um-string} the string modules 
\begin{align}\label{eq:clasif-indec-um-string}
&\Utt_{i,r},&  i \in \I_4, &\  r\geq1;
& &\Vtt_{j,t}, & &\Wtt_{j,t}, & j \in \I_{2}, &\  t \geq 0.
\end{align}
with action described in \eqref{eq:string-action-basis} and Table \ref{table:string-action};

\medbreak
\item\label{item:clasif-indec-um-band} the band modules 
$\Att_{\lambda,n}$,  $\Btt_{\lambda,n}$,  $n\geq 1$, $\lambda \in \ku^{\times}$,
with action described in \eqref{eq:band-action-basis} and Table \ref{table:band-action}.
\end{enumerate}
\end{theorem}

To prove the Theorem, we apply the functor $\Ftt$ to the indecomposable objects in $\rep(Q,I) $ where
$\Ftt:\rep(Q,I) \to \lmod{\mathfrak{u}(\mgo)}$
is given by composition of the equivalences  of categories
\begin{align}\label{equi-cate}
\xymatrix{\rep (Q, I) \ar@{->}[r]^{\sim}_{\Ftt_1} \ar@/^2pc/[rrr]^{\Ftt}
& \rmod{\mathfrak{u}(\mgo)^{\mathtt b}}\ar@{->}[r]^{\sim}_{\Ftt_2}
&\lmod{\mathfrak{u}(\mgo)^{\mathtt b}}
\ar@{->}[r]^{\sim}_{\Ftt_3}
&\lmod{\mathfrak{u}(\mgo)}, }
\end{align}
where $\Ftt_1$ is given by \cite[III.1.6]{ass} and Theorem \ref{teo-iso-quiver};
$\Ftt_2$ is induced by he anti-isomorphism $\psi$ defined above; and $\Ftt_3 = \Ind^l_{e_{\mathfrak{u}(\mgo)}}$ given  in \eqref{functor-te}.

In Subsection \ref{sub-string} we compute the string  modules in the category
$\lmod{\mathfrak{u}(\mgo)^{\mathtt b}}$ and  apply the functor $\Ind^l_{e_{\mathfrak{u}(\mgo)}}$ from  \eqref{functor-te} 
to describe the associated indecomposable in $\lmod{\mathfrak{u}(\mgo)}$. 
In Subsection \ref{sub-band} we do the same with the band modules.

\begin{remark}
Let $\Ftt$ be the functor given by \eqref{equi-cate}. Then
\[\Ftt(S(i))=V_i,\quad \Ftt(P(i))=P(V_i),\quad i\in \I_{0,1},\]
where $S(i)$ and $P(i)$ are respectively the simple and indecomposable projective objects of $\rep(Q,I)$ of index $i$. 
The proof is straightforward.
\end{remark}

\subsection{String modules} \label{sub-string}
For the sake of the reader non-specialist in representations of finite-dimensional algebras, we sketch
the calculation of string modules.
We start with some terminology. Let $(Q,I)$ be a bound quiver. 
\begin{itemize}[leftmargin=*]\renewcommand{\labelitemi}{$\circ$}
\item A {\it zero path} is a path $p$ in $Q$ that belongs to $I$. 

\medbreak
\item   A  {\it zero-relation} of $(Q,I)$  is a zero path such that none of its proper subpaths is a zero path.

\medbreak
\item  A {\it binomial relation} of $(Q,I)$ is a pair $(p,q)$ of non-zero paths with the same source $i$ and the same target $j$
such that  $\lambda p+\mu q\in I$
for some $\lambda,\mu \in \ku^{\times}$ (unique up to a scalar). Then $p,q$ are the {\it  maximal subpaths} of 
the binomial relation $(p,q)$, while   $i$ and $j$ are their {\it start-} and  {\it end-}points.

\medbreak
\item  Given  $\alpha\in Q_1$, we denote by   $\alpha^{-1}:t(\alpha)\to s(\alpha)$ the  formal inverse of $\alpha$; this extends
to paths in $Q$.

\medbreak
\item A {\it walk} of  length $r >0$ is a sequence $u=c_1\ldots c_r$, with $c_i$ an arrow or the inverse of an arrow such that $s(c_{i+1})=t(c_{i})$, for all $i \in \I_{r-1}$.
A trivial walk of a vertex $i$ is the trivial path $\varepsilon_i$. A walk $u$ is called {\it reduced} if either $u$ is trivial or $u=c_1\ldots c_r$ with $c_{i+1}\neq c_i^{-1}$, for all $i\in \I_{r-1}$.

\medbreak
\item
Let $u=c_1\ldots c_r$  be a non-trivial reduced walk in $Q$. A non-trivial path $p$ of $Q$ is \emph{contained} in $u$ if there are $ i\leq j \in \I_{r}$ such that $p=c_i\ldots c_j$ or $p^{-1}=c_i\ldots c_j$.

\medbreak
\item
A {\it string} is a reduced walk  $u=c_1\ldots c_r$ in $Q$   such that each path contained in $u$ is neither a zero-relation nor a maximal subpath of a binomial relation.
\end{itemize}

\begin{example}
Let $(Q,I)$ be   given by \eqref{quiver} and \eqref{eq:bound}. The reduced walk $\alpha_1\beta_2\alpha_1\alpha_2^{-1}$ is not a string in $(Q,I)$ since $\alpha_1\beta_2$ is a maximal subpath of the binomial relation $(\alpha_1\beta_2,\alpha_2\beta_1)$. But  $\alpha_1\alpha_2^{-1}\alpha_1$ is a string in $(Q,I)$.
\end{example}

\begin{definition} \cite[$\S 2$]{WW}
Let $u$ be  a string  in a bound quiver $(Q,I)$.
The string module $M(u)=(M_k, \varphi_{\gamma})_{k\in Q_0, \gamma \in Q_1}\in \rep (Q,I)$ is defined by the following procedure:

\begin{itemize}[leftmargin=*]\renewcommand{\labelitemi}{$\diamond$}
\item  If $u$ is the trivial path at a vertex $i$, then $M(u)$ is the simple module at $i$. 

\medbreak
\item 
Otherwise $u=c_1\ldots c_r$, with $r\geq 1$, and $c_i$ or $c_i^{-1}$ is an arrow. For each $i \in \I_{r+1}$, we consider a one-dimensional vector space $U_i\coloneqq \ku u_i$. For $i \in \I_{r}$,  $U_{c_i}: U_i\to U_{i+1}$ denotes the map sending $u_i$ to $u_{i+1}$ if $c_i$ is an arrow and $U_{c_i}: U_{i+1}\to U_{i}$ is the map sending $u_{i+1}$ to $u_i$ if $c_i^{-1}$ is an arrow. 

\medbreak
\item Now for $k \in Q_0$, $M_k$ is the direct sum of  all $U_i$'s such that $s(c_i)=k$ or $i=r+1$ if $t(c_r)=k$; here the empty direct sum is $0$. 

\medbreak
\item 
For $\gamma \in Q_1$  appearing in $u$, $\varphi_{\gamma}$ is the direct sum of the maps $U_{c_i}$ if $c_i=\gamma$ or $c_i^{-1}=\gamma$; otherwise $\varphi_{\gamma}=0$.
\end{itemize}
As explained in \cite{HL} this  is indeed a module over the bound quiver.
\end{definition}

For the rest of this subsection $(Q,I)$ is the bound quiver given by \eqref{quiver} and \eqref{eq:bound}.
In order to determine all its strings, we consider the words
\begin{align*}
s_1 &=\alpha_1\alpha^{-1}_2, & s_2 &=\alpha^{-1}_1\alpha_2,& 
s_3 &=\beta_1\beta^{-1}_2& & \text{and}&  s_4 & = \beta^{-1}_1\beta_2
\end{align*}  in the vocabulary $\{\alpha_i,\,\alpha^{-1}_i, \, \beta_i, \, \beta^{-1}_i: i\in \I_{2}\}$.

\begin{lemma}\label{strings}
The strings in $(Q,I)$ are  
the trivial paths $\varepsilon_0$, $\varepsilon_1$, the families
\begin{align}
\label{string-uvw}
\begin{aligned}
u_i(r) &= s_i^r, && & i\in \I_4, \ r&\geq 1;   
\\ 
v_1(t) &= s_1^t\alpha_1,& v_2(t) &= \alpha_2s_2^t ,&   t&\geq 0;
\\
w_1(t) &= s_3^t\beta_1,& w_2(t) &= \beta_2s_4^t, &  t&\geq 0;
\end{aligned}
\end{align}
and the inverse families
\begin{align}\label{string-uvw-inverse}
\begin{aligned}
s_1^{-r} &= \alpha_2s_2^{r-1} \alpha_1^{-1}, &s_2^{-r} &= \alpha_2^{-1} s_1^{r-1} \alpha_1, \\
s_3^{-r} &= \beta_2s_4^{r-1} \beta_1^{-1}, &s_4^{-r} &= \beta_2^{-1} s_3^{r-1} \beta_1,
&
r&\geq 1;   
\\ 
(s_1^{t}\alpha_1)^{-1} &=  s_2^{t} \alpha_1^{-1}, & (\alpha_2s_2^{t})^{-1} &= \alpha_2^{-1} s_1^{t}, \\
(s_3^{t}\beta_1)^{-1} &=  s_4^{t} \beta_1^{-1}, &(\beta_2s_4^{t})^{-1} &= \beta_2^{-1} s_3^{t}, &  
t&\geq 0.
\end{aligned}
\end{align}
\end{lemma}
\pf
Straightforward.
\epf

\begin{remark}
Given   a string $u$ as in \eqref{string-uvw}, the inverse $u^{-1}$ is also a string and $M(u) \simeq M(u^{-1})$. 
Thus  we do not need to consider the strings in \eqref{string-uvw-inverse}.
\end{remark}

\noindent \emph{Proof of Theorem \ref{thm:clasif-indec-um} \ref{item:clasif-indec-um-string}.}
We  apply the functor $\Ftt$ given by \eqref{equi-cate} to the string modules 
$M(u)$ with $u$ as in \eqref{string-uvw}. We describe below the modules
\begin{align*}
&\Utt_{i,r}\coloneqq \Ftt(M(u_i(r))),&  i \in \I_4, &\  r\geq1;
\\
&\begin{aligned}
\Vtt_{j,t} &\coloneqq \Ftt(M(v_j(t))), \\
\Wtt_{j,t} &\coloneqq \Ftt(M(w_j(t))),
\end{aligned}& j \in \I_{2}, &\  t \geq 0.
\end{align*}

For illustration, we compute $\Utt_{1,r}$ explicitly; the others are  similar.

\medbreak  
The quiver representation 
$M(u_1(r))=(M_{k},\varphi_{\gamma})_{k\in \I_{0,1}, \gamma\in Q_1}$ is given by
\begin{align*}
M_0 &=\ku\{m_{2{\ell}+1}:\ell\in \I_{0,r}\}, &
M_1 &=  \ku\{m_{2 \ell}: \ell\in\I_r\}
\\
\varphi_{\varepsilon_0}&=\id_{M_0}, \quad \varphi_{\varepsilon_1}=\id_{M_1}, &
\varphi_{\beta_j} &= 0, \quad j \in \I_2,
\\
\varphi_{\alpha_1}(m_{2\ell+1}) &= \begin{cases}
m_{2\ell+2},& \ell \in \I_{0,r-1},\\
0,& \ell=r,
\end{cases}  &  \varphi_{\alpha_2}(m_{2\ell+1}) &= \begin{cases}
0,& \ell=0,\\
m_{2\ell-1},& \ell\in \I_r.
\end{cases}
\end{align*}
Hence $\Ftt_1(M(u_1(r)))=M_0\oplus M_1 \in \rmod{\mathfrak{u}(\mgo)^{\mathtt b}}$ with action 
\begin{align*}
&m_{\ell}\cdot\varepsilon_0= \begin{cases}
m_{\ell}, &\!\!\!\ell=2k+1, k \in \I_{0,r},\\
0,&\!\!\!\!\! \text{otherwise};
\end{cases}& &m_{\ell}\cdot\alpha_1= \begin{cases}
m_{\ell+1},&\!\!\!\! \ell=2k+1,k \in \I_{0,r-1}, \\
0,&\!\!\!\! \text{otherwise};
\end{cases}&\\[.2em]
& m_{\ell}\cdot\varepsilon_1= \begin{cases}
m_{\ell},& \!\!\!\ell=2k, k \in \I_{r},\\
0,& \!\!\!\!\text{otherwise};
\end{cases} & 
&m_{\ell}\cdot\alpha_2= \begin{cases}
m_{\ell-1},&\!\!\!\! \ell=2k+1,k \in \I_{r} \\
0,& \!\!\!\!\text{otherwise}.
\end{cases}&
\end{align*}
The arrows $\beta_j$, $j\in \I_{2}$, act trivially on $\Ftt_1(M(u_1(r)))$. By definition of $\Ftt_2$,  we get that $\Ftt_2\Ftt_1(M(u_1(r)))=M_0\oplus M_1$ is a left $\mathfrak{u}(\mgo)^{\mathtt b}$-module with action 
\begin{align*}
&e_0\cdot m_{\ell}= \begin{cases}
m_{\ell},& \!\!\!\ell=2k+1, k \in \I_{0,r}\\
0,& \!\!\!\!\text{otherwise}
\end{cases}& &ae_0\cdot m_{\ell}= \begin{cases}
m_{\ell+1},&\!\!\!\! \ell=2k+1,k \in \I_{0,r-1} \\
0,&\!\!\!\! \text{otherwise}
\end{cases}&\\[.2em]
& e_1\cdot m_{\ell}= \begin{cases}
m_{\ell},&  \!\!\!\ell=2k, k \in \I_{r}\\
0,& \!\!\!\! \text{otherwise}
\end{cases} & 
&b^3e_0\cdot m_{\ell}= \begin{cases}
m_{\ell-1},& \!\!\!\! \ell=2k+1,k \in \I_{r} \\
0,& \text{otherwise}.
\end{cases}&
\end{align*}
The other elements of  the basis \eqref{eq:base-um-basica} of $\mathfrak{u}(\mgo)^{\mathtt b}$ act trivially. 

\medbreak

Finally, we apply the functor $\Ind^l_{e_{\mathfrak{u}(\mgo)}}$ and see that $\Utt_{1,r} \in \lmod{\mathfrak{u}(\mgo)}$ 
has a basis ${z}_1=e_0\otimes m_1, z_2=e_1\otimes m_2, \ldots, z_{4r+1}=e_0\otimes m_{2r+1}$. The  action on $\Utt_{1,r}$ is described in \eqref{eq:string-action-basis} and Table \ref{table:string-action}.
The calculation for the other string modules $\Utt_{i,r}$, $\Vtt_{j,s}$, $\Wtt_{j,s}$ is similar. In all cases 
there is a basis $\{z_i\}_{i \in \I_{d}}$ where  $d$ is the dimension of the module such that
the action is given by
\begin{align}\label{eq:string-action-basis}
a \cdot z_i &= \boldsymbol{\kappa}_{i} z_{i+1}, &
b \cdot z_i &= \boldsymbol{\mu}_{i} z_{i-1}, &
c \cdot z_i &= \boldsymbol{\nu}_{i} z_{i}.
\end{align}
Here $\boldsymbol{\kappa}_{i}$  and $\boldsymbol{\mu}_{i}$ take the values 0 or 1 and we specify in Table \ref{table:string-action} the $i$'s where the value is 0. The Table also records the value of $\boldsymbol{\nu}_{i}$ 
(which is again 0 or 1).

\begin{table}[ht]
\caption{String modules: coefficients in \eqref{eq:string-action-basis}}
\label{table:string-action}
\begin{tabular}{|c|c|c|c|c|}
\hline
Family & $\dim$  & $\boldsymbol{\kappa}_{i}$  & $\boldsymbol{\mu}_{i}$ & $\boldsymbol{\nu}_{i}$ 
\\ \hline
$\Utt_{1,r}$ & $4r+1$ &  $i \equiv 0 (4)$ or $i = 4r+1$   &  $i \equiv 2 (4)$   or $i = 1$    & $i+1$
\\ \hline
$\Utt_{2,r}$ & $4r+3$ &$i \equiv 3 (4)$    &   $i \equiv 1 (4)$         & $i$
\\ \hline
$\Utt_{3,r}$ & $4r+3$ & $i \equiv 0 (4)$  or $i = 4r+3$    &  $i \equiv 0 (4)$   or $i = 1$         & $i$
\\ \hline
$\Utt_{4,r}$ & $4r+1$ & $i \equiv 1 (4)$    & $i \equiv 1 (4)$           & $i+1$
\\ \hline
$\Vtt_{1,t}$ & $4(t+1)$ & $i \equiv 0 (4)$    &    $i \equiv 2 (4)$  or $i = 1$         & $i+1$
\\ \hline
$\Vtt_{2,t}$ & $4(t+1)$ & $i \equiv 3 (4)$  or $i = 4(t+1)$    &    $i \equiv 1 (4)$        & $i$
\\ \hline
$\Wtt_{1,t}$ & $4(t+1)$ & $i \equiv 0 (4)$   or $i = 4(t+1)$   &   $i \equiv 0 (4)$     or $i = 1$      & $i$
\\ \hline
$\Wtt_{2,t}$ & $4(t+1)$ & $i \equiv 1 (4)$  or $i = 4(t+1)$    &  $i \equiv 1 (4)$  or $i = 1$          & $i+1$
\\  \hline 
\end{tabular}
\end{table}

We illustrate the string modules above (for $r=t=2$) via a directed graph where the arrows above,
oriented from left to right, indicate the action of $a$ while the 
arrows below, oriented from right to left present the action of $b$, as in \eqref{graph-1}. Also, $\circ_{z_i}$ means $\boldsymbol{\nu}_{i}=0$ while $\bullet_{z_i}$ means $\boldsymbol{\nu}_{i}=1$.
\vspace{.5cm}

{\scriptsize
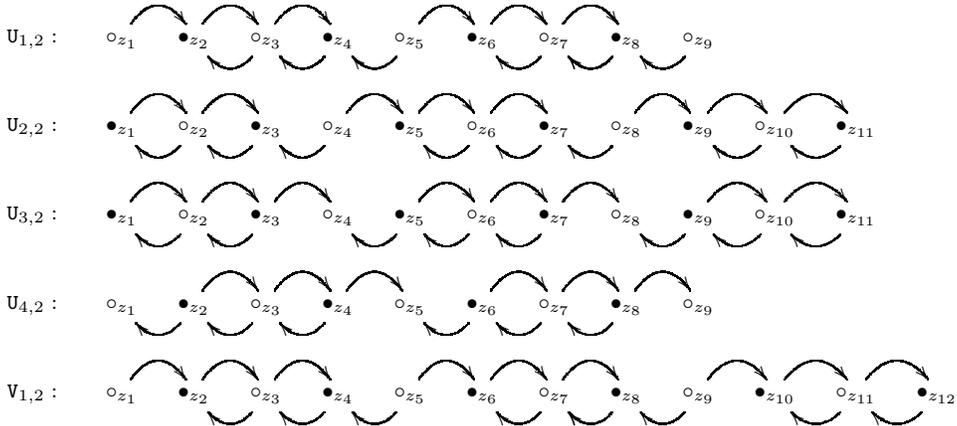
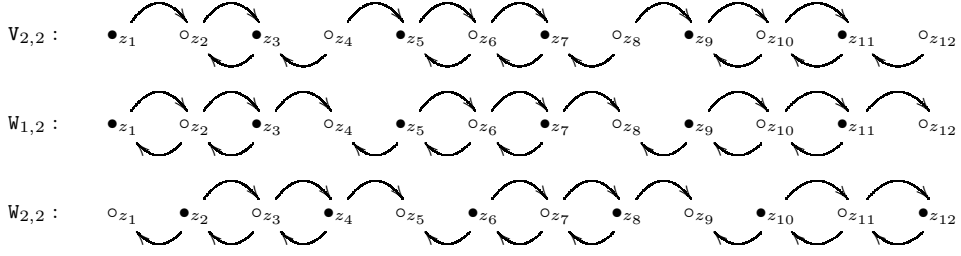
\begin{align*} 
\Utt_{1,2}:&\xymatrix@C=3mm@R=6mm{   
& \circ_{z_1}  \ar@/^1pc/[r]
& \bullet_{z_2} \ar@/^1pc/[r] 
& \circ_{z_3}  \ar@/^1pc/[r] \ar@/^1pc/[l]
& \bullet_{z_4} \ar@/^1pc/[l]
&\circ_{z_5}  \ar@/^1pc/[r]\ar@/^1pc/[l]
&\bullet_{z_6}\ar@/^1pc/[r]
&\circ_{z_7}\ar@/^1pc/[r]\ar@/^1pc/[l]
&\bullet_{z_8}\ar@/^1pc/[l]\ar@/^1pc/[l]
&\circ_{z_9}\ar@/^1pc/[l]
} & \\[.6cm]
\Utt_{2,2}:& \xymatrix@C=3mm@R=6mm{   
& \bullet_{z_1} \ar@/^1pc/[r]
& \circ_{z_2} \ar@/^1pc/[r] \ar@/^1pc/[l]
& \bullet_{z_3}\ar@/^1pc/[l]
& \circ_{z_4}\ar@/^1pc/[r] \ar@/^1pc/[l]
&\bullet_{z_5} \ar@/^1pc/[r]
&\circ_{z_6} \ar@/^1pc/[r] \ar@/^1pc/[l]
&\bullet_{z_7}\ar@/^1pc/[l]
&\circ_{z_8}\ar@/^1pc/[r]\ar@/^1pc/[l]
&\bullet_{z_9}\ar@/^1pc/[r]
&\circ_{z_{10}}\ar@/^1pc/[r]\ar@/^1pc/[l]
&\bullet_{z_{11}}\ar@/^1pc/[l]
}&\\[.6cm]
\Utt_{3,2}:& \xymatrix@C=3mm@R=6mm{   
& \bullet_{z_1} \ar@/^1pc/[r]
& \circ_{z_2} \ar@/^1pc/[r] \ar@/^1pc/[l]
& \bullet_{z_3}\ar@/^1pc/[r] \ar@/^1pc/[l]
& \circ_{z_4} 
&\bullet_{z_5} \ar@/^1pc/[r]\ar@/^1pc/[l]
&\circ_{z_6} \ar@/^1pc/[r]\ar@/^1pc/[l]
&\bullet_{z_7}\ar@/^1pc/[r]\ar@/^1pc/[l]
&\circ_{z_8}
&\bullet_{z_9}\ar@/^1pc/[r]\ar@/^1pc/[l]
&\circ_{z_{10}}\ar@/^1pc/[r]\ar@/^1pc/[l]
&\bullet_{z_{11}}\ar@/^1pc/[l]
}&\\[.6cm]
\Utt_{4,2}:&\xymatrix@C=3mm@R=6mm{   
& \circ_{z_1}  
& \bullet_{z_2} \ar@/^1pc/[r] \ar@/^1pc/[l] 
& \circ_{z_3}  \ar@/^1pc/[r] \ar@/^1pc/[l]
& \bullet_{z_4} \ar@/^1pc/[l] \ar@/^1pc/[r] 
&\circ_{z_5}  
&\bullet_{z_6}\ar@/^1pc/[r] \ar@/^1pc/[l]
&\circ_{z_7}\ar@/^1pc/[r]\ar@/^1pc/[l]
&\bullet_{z_8}\ar@/^1pc/[r]\ar@/^1pc/[l]
&\circ_{z_9}
} & \\[.6cm]
\Vtt_{1,2}:& \xymatrix@C=3mm@R=6mm{   
&\circ_{z_1} \ar@/^1pc/[r]
&\bullet_{z_2} \ar@/^1pc/[r] 
&\circ_{z_3} \ar@/^1pc/[r] \ar@/^1pc/[l]
&\bullet_{z_4} \ar@/^1pc/[l]
&\circ_{z_5} \ar@/^1pc/[r]\ar@/^1pc/[l]
&\bullet_{z_6}\ar@/^1pc/[r]
&\circ_{z_7}\ar@/^1pc/[r]\ar@/^1pc/[l]
&\bullet_{z_8}\ar@/^1pc/[l]\ar@/^1pc/[l]
&\circ_{z_9}\ar@/^1pc/[r]\ar@/^1pc/[l]
&\bullet_{z_{10}}\ar@/^1pc/[r]
&\circ_{z_{11}}\ar@/^1pc/[r]\ar@/^1pc/[l]
&\bullet_{z_{12}}\ar@/^1pc/[l]
}&\\[.6cm]
\Vtt_{2,2}:& \xymatrix@C=3mm@R=6mm{   
&\bullet_{z_1} \ar@/^1pc/[r]
&\circ_{z_2} \ar@/^1pc/[r] 
&\bullet_{z_3} \ar@/^1pc/[l]
&\circ_{z_4} \ar@/^1pc/[r]\ar@/^1pc/[l]
&\bullet_{z_5} \ar@/^1pc/[r]
&\circ_{z_6}\ar@/^1pc/[r]\ar@/^1pc/[l]
&\bullet_{z_7}\ar@/^1pc/[l]
&\circ_{z_8}\ar@/^1pc/[r]\ar@/^1pc/[l]
&\bullet_{z_9}\ar@/^1pc/[r]
&\circ_{z_{10}}\ar@/^1pc/[r]\ar@/^1pc/[l]
&\bullet_{z_{11}}\ar@/^1pc/[l]
&\circ_{z_{12}}\ar@/^1pc/[l]
}&\\[.6cm]
\Wtt_{1,2}:& \xymatrix@C=3mm@R=6mm{   
&\bullet_{z_1} \ar@/^1pc/[r]
&\circ_{z_2} \ar@/^1pc/[r] \ar@/^1pc/[l]
&\bullet_{z_3} \ar@/^1pc/[r] \ar@/^1pc/[l]
&\circ_{z_4} 
&\bullet_{z_5} \ar@/^1pc/[r]\ar@/^1pc/[l]
&\circ_{z_6}\ar@/^1pc/[r]\ar@/^1pc/[l]
&\bullet_{z_7}\ar@/^1pc/[r]\ar@/^1pc/[l]
&\circ_{z_8}
&\bullet_{z_9}\ar@/^1pc/[r]\ar@/^1pc/[l]
&\circ_{z_{10}}\ar@/^1pc/[r]\ar@/^1pc/[l]
&\bullet_{z_{11}}\ar@/^1pc/[r]\ar@/^1pc/[l]
&\circ_{z_{12}}
}&\\[.6cm]
\Wtt_{2,2}:& \xymatrix@C=3mm@R=6mm{   
&\circ_{z_1} 
&\bullet_{z_2}\ar@/^1pc/[l] \ar@/^1pc/[r] 
&\circ_{z_3} \ar@/^1pc/[r] \ar@/^1pc/[l]
&\bullet_{z_4} \ar@/^1pc/[r] \ar@/^1pc/[l]
&\circ_{z_5}
&\bullet_{z_6}\ar@/^1pc/[r]\ar@/^1pc/[l]
&\circ_{z_7}\ar@/^1pc/[r]\ar@/^1pc/[l]
&\bullet_{z_8}\ar@/^1pc/[r]\ar@/^1pc/[l]
&\circ_{z_9}
&\bullet_{z_{10}}\ar@/^1pc/[r]\ar@/^1pc/[l]
&\circ_{z_{11}}\ar@/^1pc/[r]\ar@/^1pc/[l]
&\bullet_{z_{12}}\ar@/^1pc/[l]
}&
\end{align*}}

\subsection{Band modules}\label{sub-band}
We need again some terminology. Let  $(Q,I)$ be a bound quiver. 
\begin{itemize}[leftmargin=*]\renewcommand{\labelitemi}{$\circ$}
\item A {\it reduced cycle} is a non-trivial reduced walk $u=c_1\ldots c_r$ of $Q$  if $s(c_1)=s(u)=t(u)=t(c_r)$  and $c_r\neq c_1^{-1}$. Another reduced cycle $u_1$ is {\it equivalent} to $u$ if $u_1$ or $u_1^{-1}$ is of the form $c_i\ldots c_rc_1\ldots c_{i-1}$, for some $i \in \I_{2,r}$.

\medbreak
\item   A non-trivial reduced cycle of $Q$ is called a {\it band} if each of its powers is a string and it is not a power of a string of less length. 

\end{itemize}

\begin{definition} \cite[$\S 2$]{WW}
Let $u=c_1\ldots c_r$ be a band and $\phi:V\to V$ an indecomposable automorphism of a finite-dimensional vector space $V$. 
The band module $N=N(u,\phi)=(N_k,\varphi_{\gamma})_{k\in Q_0, \gamma \in Q_1}\in \rep (Q,I)$ is defined by:

\begin{itemize}[leftmargin=*]\renewcommand{\labelitemi}{$\diamond$}
\item For each $i \in \I_{r}$, take $V_i=V$. 

\medbreak
\item Let $i \in \I_{r-1}$. If $c_i$ is an arrow, then
$f_{c_i}:V_i\to V_{i+1}$ is the identity map; otherwise $f_{c_i}:V_{i+1} \to V_i$ is also the identity map. 

\medbreak
\item For $i=r$, we have that $f_{c_r}=\phi:V_r\to V_1$ if $c_r$ is an arrow and $f_{c_r}=\phi^{-1}:V_1\to V_r$ if $c_r$ is the inverse of an arrow. 

\medbreak
\item  For each  $k \in Q_0$, if $k$ appears in $u$, then $N_k$ is the direct sum of the spaces $V_i$ such that $s(c_i)=k$, and $N_k=0$  otherwise. 

\medbreak
\item
For each arrow $\gamma$ of $Q$, if $\gamma$ appears in $u$, then $\varphi_{\gamma}$ is the direct sum of the maps $f_{c_i}$ such that $c_i=\gamma$ or $c_{i}^{-1}=\gamma$, and  $\varphi_{\gamma}=0$ otherwise. 
\end{itemize}
\end{definition}
Notice that if $u$ and $u_1$ are equivalent, then $N=N(u,\phi)$ is isomorphic to $N=N(u_1,\phi)$ in $\rep(Q,I)$.

\medbreak
Assume now that $(Q,I)$ is the bound quiver given by \eqref{quiver} and \eqref{eq:bound}.
Then the non-equivalent bands of $(Q, I)$ are:
\begin{align}\label{bands}
u=\alpha_1\alpha^{-1}_2,&  &v=\beta_1\beta^{-1}_2.
\end{align}

\noindent \emph{Proof of Theorem \ref{thm:clasif-indec-um} \ref{item:clasif-indec-um-band}.}
Let $n\geq 1$,  let $V$ be an $n$-dimensional vector space and let $\phi:V\to V$  be an indecomposable automorphism.
Then $\phi$ is represented  in some basis $B$ of $V$ by the Jordan block  $J_{n}(\lambda)$ for a fixed $\lambda \in \ku^{\times}$. 
Fix vector spaces $N_0$ and $N_1$ with bases $\{\mtt_j:j\in\I_{n}\}$ and
$\{\mtt_{n+j}:j\in\I_{n}\}$ respectively. 
Then the band module $N(u,\lambda,n) \coloneqq N(u,\phi)=(N_k,\varphi_{\gamma})_{k\in Q_0, \gamma \in Q_1}\in \rep (Q,I)$ is given by:
\begin{align*}
\varphi_{\alpha_1}(\mtt_j)=\mtt_{n+j},\,\,j\in\I_{n},& &
&\varphi_{\alpha_2}(\mtt_j)= \begin{cases}
\lambda \mtt_{n+j},& j=1,\\
\lambda \mtt_{n+j}+\mtt_{n+j-1},& j\in \I_{2,n},
\end{cases} 
\end{align*}
and $\varphi_{\beta_i}=0,$ $i \in \I_2$.
Similarly,  $N(v,\lambda,n) \coloneqq N(v,\phi) = (N_k,\varphi_{\gamma})$  is given by:
\begin{align*}
&\begin{aligned}
\varphi_{\beta_1}(\mtt_{n+j})&=\mtt_{j},& j&\in\I_{n}, \\
\varphi_{\alpha_i}&=0,&  i &\in \I_2;
\end{aligned}
&
\varphi_{\beta_2}(\mtt_{n+j}) &= \begin{cases}
\lambda \mtt_{j},& j=1,\\
\lambda \mtt_{j}+\mtt_{j-1},& j\in \I_{2,n}.
\end{cases}
\end{align*}
Let $\Att_{\lambda,n}\coloneqq \Ftt(N(u,\lambda,n))$ and $\Btt_{\lambda,n}\coloneqq \Ftt(N(v,\lambda,n))$, where $n\geq 1$ and $\lambda\in \ku^{\times}$.  Then $\dim \Att_{\lambda,n}=\dim \Btt_{\lambda,n}=4n$ 
and arguing as above, we see that there are bases of  $\Att_{\lambda,n}$ and $\Btt_{\lambda,n}$, both denoted by $\{z_i: i \in \I_{1,4n}\}$, where  $\mathfrak{u}(\mgo)$ acts by 
\begin{align}\label{eq:band-action-basis}
a \cdot z_i &= \boldsymbol{\kappa}_{i} z_{i+1}, &
b \cdot z_i &= \boldsymbol{\mu}_{i} z_{i-1} + \boldsymbol{\xi}_i \lambda z_{i+3}, &
c \cdot z_i &= \boldsymbol{\nu}_{i} z_{i}.
\end{align}
Here $\boldsymbol{\kappa}_{i}$,  $\boldsymbol{\mu}_{i}$  and $\boldsymbol{\xi}_{i}$ take the values 0 or 1; 
we specify in Table \ref{table:band-action} the $i$'s where the value is 0. The Table also records the value of $\boldsymbol{\nu}_{i}$ 
(which is again 0 or 1).

\begin{table}[ht]
\caption{Band modules: coefficients in \eqref{eq:band-action-basis}}
\label{table:band-action}
\begin{tabular}{|c|c|c|c|c|}
\hline
Family &  $\boldsymbol{\kappa}_{i}$  & $\boldsymbol{\mu}_{i}$ & $\boldsymbol{\xi}_{i}$ & $\boldsymbol{\nu}_{i}$ 
\\ \hline
$\Att_{\lambda,n}$   &  $i \equiv 0 (4)$    &  $i \equiv 2 (4)$ or $i=1$  & $i \equiv 0, 2, 3 (4)$  & $i+1$
\\ \hline
$\Btt_{\lambda,n}$   &$i \equiv 0 (4)$    &   $i \equiv 0 (4)$ or $i=1$ &  $i \equiv 0, 2, 3 (4)$     & $i$
\\ \hline
\end{tabular} 
\end{table} 

We illustrate the band modules above (for $n=2$) via a directed graph where 
 the arrows above, oriented from left to right, indicate the action of $a$ while the 
arrows below indicate the actions of $b$, as in \eqref{graph-1}.
Here we have that $bz_1=\lambda z_{4}$ and $bz_5=z_4+\lambda z_{8}$; thus we have labeled arrows
below the diagram from left to right.
\vspace{.3cm}

{\scriptsize
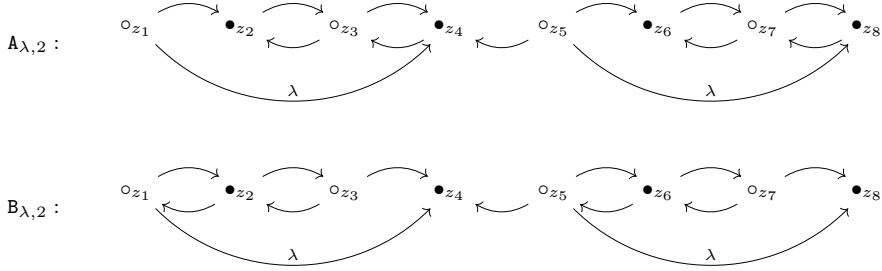
\[\Att_{\lambda,2}:\qquad 
\begin{tikzcd}
\circ_{z_1} \arrow[r,bend left=30] \arrow[rrr,bend right=45,"\lambda"] &
\bullet_{z_2}\arrow[r,bend left=30] &
\circ_{z_3} \arrow[r,bend left=30] \arrow[l,bend left=30]&
\bullet_{z_4}\arrow[l,bend left=30]& 
\circ_{z_5}\arrow[r,bend left=30]\arrow[l,bend left=30] \arrow[rrr,bend right=45,"\lambda"]&
\bullet_{z_6}\arrow[r,bend left=30]&
\circ_{z_7}\arrow[r,bend left=30]\arrow[l,bend left=30] &
\bullet_{z_8}\arrow[l,bend left=30]\\
\end{tikzcd}
\]
\vspace{.3cm}
\[\Btt_{\lambda,2}:\qquad 
\begin{tikzcd}
\circ_{z_1} \arrow[r,bend left=30] \arrow[rrr,bend right=45,"\lambda"] &
\bullet_{z_2}\arrow[r,bend left=30]\arrow[l,bend left=30] &
\circ_{z_3} \arrow[r,bend left=30] \arrow[l,bend left=30]&
\bullet_{z_4}& 
\circ_{z_5}\arrow[r,bend left=30]\arrow[l,bend left=30] \arrow[rrr,bend right=45,"\lambda"]&
\bullet_{z_6}\arrow[r,bend left=30] \arrow[l,bend left=30]&
\circ_{z_7}\arrow[r,bend left=30]\arrow[l,bend left=30] &
\bullet_{z_8}\\
\end{tikzcd}
\]}

\begin{remark}\label{dualstring}
If $H$ is a Hopf algebra and $V\in \lmod{H}$ is indecomposable, then so is $V^{\ast}$.  Recall that the action of $H$ on $V^{\ast}$ is given by $(h\cdot \varphi)(v)=\varphi(\Ss(h)\cdot v)$, where $\Ss$ is the antipode map of $H$, $h\in H$, $\varphi\in V^{\ast}$ and $v\in V$. By direct calculations, we obtain that the duals of the string and band modules over $\mathfrak{u}(\mgo)$  are 
\begin{align}\label{eq:dual}
\Utt_{1,r}^{\ast} &\simeq \Utt_{4,r}, & 
\Utt_{2,r}^{\ast} &\simeq   \Utt_{3,r}, &
\Vtt_{1,t} ^{\ast} &\simeq   \Wtt_{1,t},   &
\Vtt_{2,t} ^{\ast} &\simeq   \Wtt_{2,t},   &
\Att_{\lambda,n}^{\ast}  &\simeq   \Btt_{\lambda,n}.
\end{align}
In fact, consider the dual basis $\{z_1^{\ast},\ldots,z_{4r+1}^{\ast}\}$ of $\Utt_{1,r}^{\ast}$. It is easy to see that

$\!\!\!\!a\cdot z^{\ast}_{i}=\begin{cases}
0,&\!\! i=4j+1,\,j\in\I_{0,r},\\
	z^{\ast}_{i-1},&\!\! \text{otherwise, }\\
\end{cases}$\qquad $b\cdot z^{\ast}_{i}=\begin{cases}
0,&\!\! i=4j+1,\,\,j\in\I_{0,r},\\
z^{\ast}_{i+1},&\!\! \text{otherwise, }
\end{cases} $ 
and $c\cdot z^{\ast}_{i}=	z^{\ast}_{i}$ whenever $i$ is even and $c\cdot z^{\ast}_{i}=0$ if $i$ is odd. Thus, reordering the dual basis in the form $\{z_{4r+1}^{\ast},\ldots,z_{1}^{\ast}\}$ we obtain that $\Utt_{1,r}^{\ast}\simeq \Utt_{4,r}$. Similarly, we prove the other isomorphisms of \eqref{eq:dual}. 
\end{remark}

\begin{remark}
If $\theta \in \Aut_{\operatorname{Lie}} (\mgo) \simeq 
\Aut_{\operatorname{Hopf}} (\mathfrak{u}(\mgo))$ and $V\in \lmod{\mathfrak{u}(\mgo)}$, 
then $V^{\theta}:=V$ with the new action $u\star v=\theta(u)\cdot v$, 
where $\cdot$ is the old action,  $u\in \mathfrak{u}(\mgo)$ and $z\in V$, is an
 $\mathfrak{u}(\mgo)$-module. Thus $\Aut (\mgo)$ acts on $\Indec \mathfrak{u}(\mgo)$. 
As an illustration, let $\theta\in \Aut (\mgo)$ be the Chevalley involution given by $\theta(a)=b$, $\theta(b)=a$ and $\theta(c)=c$. Then $\Att^{\theta}_{\lambda,2}\simeq\Att_{1/\lambda,2}$ by a suitable change of basis. 
\end{remark}

\section{A Drinfeld Double of the Jordan Plane}\label{sec:double-Jordan}
In this Section we define an infinite-dimensional Hopf algebra $\widetilde{D}$ such that $\widetilde{D} \twoheadrightarrow D(H)$ 
and describe the commutative diagram of Hopf algebras \eqref{squarem}. The material here is similar to 
\cite[\S 2]{ap} and \cite[\S 4]{ap1} for the Jordan and super Jordan planes in odd characteristic, respectively. 

\subsection{The Hopf algebra   \texorpdfstring{$\widetilde{D}$}{}} We start by defining two Hopf algebras $\widetilde{H}$ and $\widetilde{K}$ such that $\widetilde{H} \twoheadrightarrow H$ and $\widetilde{K} \twoheadrightarrow K\coloneqq  H^{\ast\,\rm{op}}$. Then $\widetilde{D}\coloneqq  \widetilde{H} \bowtie_{\sigma} \widetilde{K}$ for a suitable 2-cocycle $\sigma$ as in \cite{dt}.
Let $\widetilde{\Gamma}=\langle \bf{g}\rangle \simeq \mathbb{Z}$. By the same formulas as in \eqref{YD-structure}, the vector space $V\simeq \cV(1,2)$ becomes a Yetter-Drinfeld module over $\ku \widetilde{\Gamma}$. 
We consider the following braided Hopf algebra and its bosonization:
\begin{align*}
\widetilde{\toba}=T(V)/\langle x_1^2, x_2^2x_1+x_1x_2^2+x_1x_2x_1\rangle  \in \yd{\ku \widetilde\Gamma}, & & \widetilde{H}= \widetilde{\toba}\# \ku \widetilde{\Gamma}.
\end{align*}

Let $x_{21}\coloneqq x_2x_1+x_1x_2\in \widetilde{\toba}$. 
Notice that  $\widetilde{\toba}/\langle x_2^4\rangle \simeq \toba(V)$.

\begin{lemma}\label{lema-Htilde}
The algebra $\widetilde{H}$ is generated by  $x_1,x_2,\bf{g}^{\pm}$ with relations
\begin{align}\label{rel-Htilde}
\begin{aligned}
&x_1^2=0, & &  x_2^2x_1=x_1x_2^2+x_1x_2x_1,\\ 
&{\bf g}x_1=x_1{\bf g}, & &  {\bf g}x_2=x_2{\bf g}+x_1{\bf g},& &{\bf g}^{\pm} {\bf g}^{\mp}=1.
\end{aligned}
\end{align}
The coproduct of $\widetilde{H}$ is given as in \eqref{for-copro-boson}. 
Here is a PBW-basis of $\widetilde{H}$:
$$B=\{x_1^{l}x_{21}^{m_1}x_2^{m_2}{\bf g}^{n}\,:\, l\in \I_{0,1},\, m_1,m_2 \in \mathbb{N}_0,\, n\in \mathbb{Z} \}.$$
\end{lemma}
\pf
It is clear that $B$ is a PBW-basis of $\widetilde{H}$. Let $A$ be the algebra presented by relations \eqref{rel-Htilde}. Using that the relations are valid in $\widetilde{H}$, we obtain an epimorphism of Hopf algebras $A \twoheadrightarrow \widetilde{H}$. Since $A$ is linearly generated by $B$ and its image by the morphism is linear independent, the result follows.
\epf

\begin{lemma}\label{commutation}
The following relations hold in $\widetilde{H}$, for $m,n\in \N_0$:
\begin{align}
\label{eq-1} &(x_1+x_2)^{2n}=x_2^{2n}+nx_{21}x_2^{2n-2},\\[.3em]
\label{eq0}  &(x_1+x_2)^{2n+1}=x_2^{2n+1}+x_1x_2^{2n}+nx_{21}x_2^{2n-1},\\[.3em]
\label{eq1} &x_{21}x_1=x_1x_{21},\qquad\qquad\qquad\qquad\qquad\,\,\, x_2^mx_{21}^{2n+1}=x_{21}^{2n+1}(x_1+x_2)^m,\\[.3em]
\label{eq2} & x_2^{2n}x_1=x_1x_2^{2n}+nx_1x_{21}x_2^{2n-2}, \qquad\qquad\! x_2^{m}x_{21}^{2n}=x_{21}^{2n}x_2^{m},\\[.3em]
\label{eq3} &x_2^{2n+1}x_1=x_1x_2^{2n+1}+nx_{21}\big(x_2^{2n}+x_1x_2^{2n-1}+x_{21}x_2^{2n-2}\big)+(n+1)x_{21}^{2n},\\[.3em]
\label{eq4} &{\bf g}x_1=x_1{\bf g},\qquad\qquad\qquad\qquad\qquad\quad\quad{\bf g}x_{21}=x_{21}{\bf g},\\[.3em]
\label{eq5}&{\bf g}^{2m+1}x_2^{n}=(x_1+x_2)^{n}{\bf g}^{2m+1},\qquad\qquad\,\, {\bf g}^{2m}x_2^{n}=x_2^{n}{\bf g}^{2m}.
\end{align}
\end{lemma}
\pf
By induction on $m,n$.
\epf

Let $\ku[\zeta]$ be the polynomial algebra which is a Hopf algebra with $\zeta$ pri\-mitive. Observe that we have a Hopf algebra epimorphism $\ku[\zeta] \twoheadrightarrow \ku^{\Gamma}$, where $\Gamma=\langle g\rangle$ and $g^2=1$. By \eqref{struc-yd-dual}, the vector space $W$ generated by $w_1,w_2$ is a Yetter-Drinfeld module over $\ku[\zeta]$. So we consider 
\begin{align*}
\widetilde{\mathfrak{B}}=T(W)/(w_1^2, w_2^2w_1+w_1w_2^2+w_1w_2w_1) \in \yd{\ku[\zeta]}, & & \widetilde{K}= (\widetilde{\mathfrak{B}} \# \ku [\zeta])^{\rm op}.
\end{align*}

\begin{lemma}
The algebra $\widetilde{K}$ is presented by generators $w_1,w_2, \zeta$ and relations
\begin{align}\label{rel-Ktilde}
\begin{aligned}
& w_1^2=0, & & w_2^2w_1=w_1w_2^2+w_1w_2w_1,\\ 
& w_1\zeta =(1+\zeta)w_1, & & w_2\zeta=(1+\zeta)w_2.
\end{aligned}
\end{align}
The coproduct of $\widetilde{K}$ is given as in \eqref{for-copro-dual-boson}. The set 
\[\{\zeta^{m_1}w_1^{l}w_{21}^{m_2}w_2^{m_3}\,:\, l\in \I_{0,1},\, m_i \in \mathbb{N}_0,\, i \in \I_3 \}\] is a PBW-basis of $\widetilde{K}$.
\end{lemma}
\pf 
It is similar to Lemma \ref{lema-Htilde}.
\epf

The algebras $\widetilde{\mathfrak{B}}$ and $\widetilde{\toba}$ are isomorphic. Hence the relations \eqref{eq1}, \eqref{eq2}, \eqref{eq3} are valid in $\widetilde{\mathfrak{B}}$ with $w_i$ instead $x_i$, $i \in \I_2$, and $w_{21}\coloneqq w_2w_1+w_1w_2$ instead $x_{21}$. However $\widetilde{\mathfrak{B}}$ and $\widetilde{\toba}$ are not isomorphic as coalgebras.

\begin{lemma}\label{lem-relations-dual}
The following relations hold in $\widetilde{K}$, for all $n \in \N_0, i \in \I_2$:
\begin{align}
\label{eq6} & w_i\zeta^n=(1+\zeta)^nw_i, & &   w_{21}\zeta=\zeta w_{21},\\
\label{eq7} &  w_2^{2n}\zeta =\zeta w_2^{2n} , & &  w_2^{2n+1}\zeta =(1+\zeta)w_2^{2n+1}.
\end{align}
\end{lemma}
\pf
It follows by induction on $n$.
\epf

In order to define $\widetilde{D}$, we need a skew-pairing $\tau$ between $\widetilde{H}$ and $\widetilde{K}$ as in \cite{dt}, that is, a linear map $\tau: \widetilde{H} \otimes \widetilde{K} \to \ku$ satisfying 
\begin{align}\label{s-p}
\begin{aligned}
&\tau(h\widetilde{h}\otimes k)=\tau(h \otimes k_{(1)})\tau(\widetilde{h} \otimes k_{(2)}),& &  \tau(1 \otimes k)=\varepsilon(k),\\
&\tau(h \otimes \widetilde{k}k)=\tau(h_{(1)}\otimes k) \tau(h_{(2)}\otimes \widetilde{k}),& &  \tau(h \otimes 1)=\varepsilon(h),
\end{aligned}
\end{align}
for all $h, \widetilde{h} \in \widetilde{H}$, $k, \widetilde{k} \in \widetilde{K}$. Then $\tau$ is convolution invertible with inverse $\tau^{-1}(h,k)=\tau(\Ss(h),k)$, $h \in \widetilde{H},\, k \in \widetilde{K}$. Thus a skew-pairing $\tau$ is equivalent to a Hopf algebra map $\phi$ from $\widetilde{H}^{\rm{cop}}$ to the Sweedler dual $\widetilde{K}^{\circ}$ of $\widetilde{K}$.

\begin{lemma}
There exists a unique skew-pairing $\tau: \widetilde{H} \otimes \widetilde{K} \to \ku$ such that
\begin{align*}
&\tau(x_1 \otimes w_i)=\delta_{i,2}, & & \tau(x_2 \otimes w_i)=\delta_{i,1}, & & \tau({\bf g}^{\pm} \otimes w_i)=0,\,\, i \in \I_2,\\
&\tau(x_i \otimes \zeta)=0,& & \tau(x_2 \otimes \zeta)=0, & & \tau({\bf g}^{\pm} \otimes \zeta )= 1. & &
\qed
\end{align*}
\end{lemma}

Let $\widetilde{D}=A_{\sigma}=  \widetilde{H} \bowtie_{\sigma} \widetilde{K}$ be  the Hopf algebra $A = \widetilde{H} \otimes \widetilde{K}$ twisted by $\sigma$, where 
$\sigma: A \otimes A \to \ku$ associated to $\tau$ is given by:
\begin{align*}
\sigma(a \otimes b, c \otimes d)= \varepsilon(a) \varepsilon(d)\tau(c \otimes b),& & a,b,c,d \in A.
\end{align*}
The product in $\widetilde{D}$ is given by 
\begin{align*}
(a \otimes b) \cdot (c \otimes d)=a \tau(c_{(1)},b_{(1)})c_{(2)}\otimes b_{(2)}\tau^{-1}(c_{(3)}, b_{(3)})d.
\end{align*}
The proof of the next result is similar to \cite[Proposition 2.3]{ap}.

\begin{prop} \label{Dtil}
\begin{enumerate}[leftmargin=*,label=\rm{(\roman*)}]
\item The algebra $\widetilde{D}$ is presented by generators $x_1$, $x_2$, ${\bf g}^{\pm}$, $w_1$, $w_2$, $\zeta$ with relations \eqref{rel-Htilde}, \eqref{rel-Ktilde} and
\begin{align*}
&w_1x_1=x_1w_1,& & w_1x_2=x_2w_1+1 +{\bf g},& & w_1{\bf g}={\bf g}w_1,\\
& w_2x_1=x_1w+1 +{\bf g},& & w_2x_2=x_2w+{\bf g}\zeta ,& & w_2{\bf g}={\bf g}w,\\
& \zeta x_i=x_i\zeta+x_i, & & i\in I_2, & & \zeta {\bf g}= {\bf g} \zeta,
\end{align*}
where ${\bf h}=1 +{\bf g}$, $w=w_1+w_2$.
The comultiplication is given by \eqref{for-copro-boson} and \eqref{for-copro-dual-boson}.
\item The following family is a PBW-basis of $\widetilde{D}$:
\begin{align}\label{PBW_basis} \{ x_1^{l_1}x_{21}^{m_1}x_2^{m_2}{\bf g}^{n}\zeta^{m_3}w_1^{l_2}w_{21}^{m_4}w_2^{m_5}\!\!:l_i \!\in \!\I_{0,1}, i \!\in \! \I_2, m_j\!\in\! \N_{0}, j\!\in \! \I_5, n\!\in \! \mathbb{Z}\}.
\end{align}

\item There exists a Hopf algebra epimorphism $\widetilde{D} \twoheadrightarrow D(H)$.

\item $\Ss^4=\id$, where $\Ss$ is the antipode map of $\widetilde{D}$. Thus, $\Ss$ is bijective. \qed
\end{enumerate} 
\end{prop}

\begin{lemma}\label{commutationDtil}
The relations \eqref{eq-1}, \dots, \eqref{eq5}, \eqref{eq6}, \eqref{eq7}, and 
\begin{align*}
&\zeta x_{21}=x_{21} \zeta,& &w_1x_{21}=x_{21}w_1,& &w_{21}x_1=x_1w_{21},& &w_{21}{\bf g}={\bf g}w_{21},&
\end{align*}
are valid in $\widetilde{D}$. Also, taking ${\bf h}= 1+{\bf g}$ and $w=w_1+w_2$, the following commutation relations hold in  $\widetilde{D}$, for all $m,n \in \N_0$:
\begin{align*}
&\zeta^nx_1=x_1\xi^n,  & &  \zeta^m x_2^{n}=x_2^{n}\sum_{j=0}^{m}{{m}\choose{j}}\xi^j,\\
& w_1x_2^{2n}=x_2^{2n}w_1+nx_1x_2^{2n-2}{\bf g}, & & w_1x_2^{2n+1}=x_2^{2n+1}w_1+x_2^{2n}{\bf h}+nx_1x_2^{2n-1}{\bf g},
\end{align*}
\vspace{-0.7cm}
\begin{align*}
& w_{21}x_2^{2n}= x_2^{2n}w_{21}+nx_2^{2n-2}(x_1{\bf g}w_1+{\bf h}^2),\\
&w_{21}x_2^{2n+1}=x_2^{2n+1}w_{21}+x_2^{2n}{\bf h}w_1
+n(x_2^{2n-1}{\bf h}^2+x_1(x_2^{2n-1}{\bf g}w_1+x_2^{2n-2}{\bf g}{\bf h})),
\end{align*}
\vspace{-0.7cm}
\begin{align*}
&w_{21}^{2n}x_2= x_2w_{21}^{2n},& &w_{21}^{2n+1}x_2= x_2w_{21}^{2n+1}+{\bf h}w_1w_{21}^{2n}\\
&w_2^{2n}x_1= x_1w^{2n}\!+\!nw_1w_2^{2n-2}, & &w_2^{2n+1}x_1= x_1w^{2n+1}\!\!+\!(n+{\bf h})w^{2n}\!+\!nww_2^{2n-1},\\
&w_2x_{21}^{2m}=x_{21}^{2m}w_2,& &
w_2x_{21}^{2m+1}=x_{21}^{2m+1}w_2+x_{21}^{2m}{\bf h}x_1,
\end{align*}
\vspace{-0.7cm}
\begin{align*}
&w_2^{2n}x_{21}= x_{21}w_2^{2n}+n(x_1w_1+{\bf h}^2)w_2^{2n-2},\\ &w_2^{2n+1}x_{21}=x_{21}w_2^{2n+1}+x_1{\bf h} w_2^{2n}
+n(x_1(w_1w_2+w_{21})+{\bf h}w_1+{\bf h}^2w_2)w_2^{2n-1},\\
&w_2x_2^{2n}=x_2^{2n}w_2+n(x_2^{2n-1}+x_2^{2n-2}x_1{\bf g}\xi ),\\ &w_2x_2^{2n+1}=x_2^{2n+1}w+x_2^{2m}{\bf g}\zeta+n(x_2^{2m}+x_1x_2^{2m-1}{\bf g}\zeta),\\
& w_2^{2n}x_2=x_2w_2^{2n}+n(x_2w_{21}+{\bf g}w)w_2^{2n-2},\\ &w_2^{2n+1}x_2=x_2w_2^{2n+1}\!\!+\!(x_2w_1\!+{\bf g}\zeta)w_2^{2n}\!\!
+\!n({\bf g}\xi w_{21}w_2^{2n-2}\!\!+\!x_2w_{21}w_2^{2n-1}\!\!+\!{\bf g}w_2^{2n}),\\
&w_2^m {\bf g}^{2n}={\bf g}^{2n}w_2^m, \qquad \qquad \qquad w_2^m {\bf g}^{2n+1}= {\bf g}^{2n+1}w^m.
\end{align*}
\end{lemma}
\begin{proof}
It follows from Proposition \ref{Dtil} and induction on $m,n$. 
\end{proof}

\subsection{\texorpdfstring{$\widetilde{D}$}{} as an extension}
Consider the algebraic group $${\bf B}\coloneqq  ((\mathbf{G}_a \times \mathbf{G}_a) \rtimes  \mathbf{G}_m) \times {\bf H}_3,$$
where $(\mathbf{G}_a \times \mathbf{G}_a) \rtimes  \mathbf{G}_m$  is the semidirect product with action of  $\mathbf{G}_m$ on $\mathbf{G}_a \times \mathbf{G}_a$ given by $\lambda \cdot (t_1, t_2) = (\lambda^2 t_1, \lambda^2 t_2)$, $\lambda \in \ku^{\times}, t_1,t_2 \in \ku$, and $\bf{H}_3$ is the Heisenberg group  of upper triangular matrices with ones in the diagonal.

\begin{prop} \label{middlecolumn}
Let $N$ be the subalgebra of $\widetilde{D}$ generated by $x_2^4$, $x_{21}^{2}$, ${\bf g}^2$, $w_2^4$, $w_{21}^2$ and $\zeta^{(2)}=\zeta^2+\zeta$. Then
\begin{enumerate}[leftmargin=*,label=\rm{(\roman*)}]
\item $N$ is a normal commutative Hopf subalgebra of $\widetilde{D}$.
\item $\widetilde{D}$ is a finitely generated free $N$-module.
\item  $N \overset{\iota}\hookrightarrow \widetilde{D} \overset{\pi}\twoheadrightarrow D(H)$ is a short exact sequence of Hopf algebras.
\item $N \simeq \ku[T^{\pm}] \otimes \ku[X_1, \cdots, X_5]$ as an algebra.
\item $N \simeq \cO({\bf{B}})$ as Hopf algebras.
\end{enumerate}
\end{prop}
\pf Using  Propositions \ref{prop:double} and \ref{Dtil},  Lemma \ref{commutationDtil} and following the same steps as in the proof of \cite[Proposition 2.6]{ap}, we obtain the result.
\epf

\subsection{Another extension}
Let $\pi:  \widetilde{D} \to U(\mfG)$ be  the Hopf algebra map given by 
\begin{align*}
\pi(x_1)&=0, &\pi(x_2)&=a, &\pi({\bf g})&=1,& \pi(w_1)&=0,& \pi(w_2)&=b, & \pi(\zeta)&=c.
\end{align*}
 Consider the subalgebra  $C$  of $ \widetilde{D}$ generated by $x_1, x_{21}, {\bf g}, w_1$ and $w_{21}$. By Lemma \ref{commutationDtil}, $C$ is a normal Hopf subalgebra of $\widetilde{D}$.

\medbreak
For our next statement we need to introduce some more notations. Let ${\mathfrak G}$ be the  
group scheme such that its algebra of functions is
the commutative Hopf algebra $\cO(\mathfrak G) \coloneqq  \ku[X_1,X_2, T^{\pm}] \otimes \Lambda(Y_1,Y_2)$ with comultiplication
\begin{align*}
&\Delta(X_1)=X_1\otimes 1 + T^2 \otimes X_1+Y_1T\otimes Y_1, & & \Delta(Y_1)=Y_1\otimes 1 + T \otimes Y_1,\\
&\Delta(X_2)=X_2\otimes 1 + 1 \otimes X_2+Y_2 \otimes Y_2, & & \Delta(Y_2)=Y_2\otimes 1 + 1 \otimes Y_2,\\
&\Delta(T)=T\otimes T.
\end{align*}

\begin{prop} \label{middlerow}
\emph{(i)} There is an isomorphism $C \simeq \cO(\mathfrak G)$ of Hopf algebras.

\emph{(ii)}  The sequence of Hopf algebra maps $ \cO(\mathfrak G) \overset{\iota}\hookrightarrow \widetilde{D} \overset{\pi}\twoheadrightarrow U(\mfG)$ is exact.
\end{prop}
\pf Similar to the proof of \cite[Proposition 2.7]{ap}.
\epf

\subsection{A commutative square} 
Let $\mathbf{G}\coloneqq (\mathbf{G}_a \times \mathbf{G}_a )\rtimes \mathbf{G}_m$ be the semidirect
product with action of $\mathbf{G}_m$ on $\mathbf{G}_a \times \mathbf{G}_a$ given by $\lambda \cdot (t_1, t_2) = (\lambda^2 t_1, t_2)$,  $\lambda \in \ku^{\times}$, $t_1, t_2 \in \ku$.
Then  $\cO(\mathbf{G})$ is isomorphic to $A\coloneqq \ku[X_1,X_2,T^{\pm}]$, where $T \in G(A)$, $X_1 \in \mathcal P_{1,T^2}(A)$, $X_2 \in \mathcal P(A)$.\medbreak

It is straightforward to check the following.

\begin{prop}\label{leftcolumn} There is a short exact sequence of Hopf algebras $$\cO(\mathbf{G}) \overset{\iota}\hookrightarrow  \cO(\mathfrak{G}) \overset{\pi}\twoheadrightarrow \nucleo$$
where the Hopf maps $\iota$ and $\pi$ are given by
\begin{align*}
\iota(X_1) &= X_1^2, &  & \iota(X_2)=X_2^2, & & \iota(T)= T^2,&&\\
\pi(X_1)= x_{21}, \quad \pi(X_2) &= w_{21}, & & \pi(T)=g, & & \pi(Y_1)=x_1, & & \pi(Y_2)=w_1. \qed
\end{align*} 
\end{prop}

We can summarize the relations between the Hopf algebras studied in the previous sections in the following diagram:

\begin{align}\label{squarem}
\begin{aligned}\xymatrix{ \cO(\mathbf{G}) \ar@{^{(}->}[r] \ar@{^{(}->}[d] & \cO({\bf B}) \ar@{->>}[r] \ar@{^{(}->}[d]& \cO({\bf G}_a^3) \ar@{^{(}->}[d]\\
\cO(\mathfrak{G}) \ar@{^{(}->}[r] \ar@{->>}[d] &   \widetilde{D} \ar@{->>}[r] \ar@{->>}[d] & U(\mfG) \ar@{->>}[d]\\
\nucleo\ar@{^{(}->}[r] &  D(H) \ar@{->>}[r] & \mathfrak{u}(\mgo).} \end{aligned} \end{align}

\begin{prop}
All columns and rows in \eqref{squarem} are exact sequences.
\end{prop} 

\pf 
The middle row is exact by Proposition \ref{middlerow} while the bottom row is exact by Theorem \ref{teo:sequence}. Also the middle and left columns are exact respectively by Proposition \ref{middlecolumn} and Proposition \ref{leftcolumn}. 
Arguing as in \cite[Proposition 2.8]{ap}, we  verify that the others sequences are exact. \epf 



\noindent \emph{Acknowledgements.} We would like to thank A. Premet for his valuable suggestions on the algebra $\mathfrak{u}(\mgo)$. We also thank E. R. Alvares for his  advice on the representation theory of basic algebras,
and M. Guerreiro and A. Grishkov for useful email exchanges.

\end{document}